\documentclass[3p,times]{elsarticle}

\usepackage{xcolor}
\usepackage{amsmath,amsthm,amsfonts,amssymb,bbm,euscript,mathrsfs}
\usepackage{enumerate,enumitem,mathabx}
\usepackage[english]{babel}
\usepackage[utf8]{inputenc}
\usepackage{dsfont}
\usepackage{mathtools}
\usepackage{nccmath}
\usepackage[normalem]{ulem}
\usepackage{hyperref}

\DeclareUnicodeCharacter{00A0}{~}

\biboptions{numbers,sort&compress}

\newtheorem{assumption}{Assumption}
\newtheorem{prop}{Proposition}
\newtheorem{example}{Example}
\newtheorem{definition}{Definition}
\newtheorem{theorem}{Theorem}
\newtheorem{coro}{Corollary}
\newtheorem{lemma}{Lemma}
\newtheorem{remark}{Remark}

\newcommand{\eqdef}{:=}
\newcommand{\un}{\mathds{1}}

\newcommand{\bP}{{\mathbb P}}
\newcommand{\bE}{{\mathbb E}}
\newcommand{\bR}{{\mathbb R}}
\newcommand{\bQ}{{\mathbb Q}}

\newcommand{\bN}{{\mathbb N}}

\newcommand{\cA}{{\mathcal A}}

\newcommand{\cB}{{\mathcal B}}

\newcommand{\cC}{{\mathcal C}}
\newcommand{\cD}{{\mathcal D}}

\newcommand{\cP}{{\mathcal P}}

\newcommand{\cN}{{\mathcal N}}
\newcommand{\cT}{{\mathcal T}}
\newcommand{\cF}{{\mathcal F}}

\newcommand{\bs}{\boldsymbol}

\DeclareMathOperator{\supp}{supp}

\newcommand{\inX}{X}
\newcommand{\inY}{Y}

\DeclareMathOperator{\graph}{gr}
\DeclareMathOperator{\crit}{crit}

\newcommand{\argmin}{\mathop{\rm argmin}}

\newcommand{\BEQ}{\begin{equation}}
\newcommand{\EEQ}{\end{equation}}

\newcommand{\conv}{\operatorname{conv}}

\newcommand{\pbest}{\operatorname{pbest}}
\newcommand{\gbest}{\operatorname{gbest}}

\begin{document}
\begin{frontmatter}

\title{A continuous-time limit for particle swarm optimization}

\author{Pascal Bianchi}
\author{Radu-Alexandru Dragomir}
\author{Yerkin Yesbay}

\address{LTCI, Télécom Paris, Institut Polytechnique de Paris, France}

\begin{abstract}
  Particle swarm optimization (PSO) is a popular class of gradient-free methods for global optimization.  
   We study a continuous-time model for momentum-free PSO with personal-best  memory and global-best
  interaction. Unlike previous models which rely on smooth approximations, we keep the original discontinuous personal- and global-best processes. We study the small-step limit and show that weak limit points satisfy a stochastic
  differential inclusion: personal-best variables obey a relaxed
  record condition, while fast switching of the global-best selector
  leads to a convexified interaction.  We then give conditions under which the true original
  selectors are recovered. Under suitable regularity and nondegeneracy
  assumptions, Brownian noise allows us to identify the personal-best
  processes with the latest record positions of the limiting trajectories and
  to prove that the particle attaining the global best is unique at
  almost every time. The resulting continuous-time dynamics retain
  both selection rules without smoothing them.
\end{abstract}


\end{frontmatter}

\section{Introduction}
We consider the problem of minimizing a non-convex continuous function
$f:\bR^d\to\bR$ without access to derivatives.
Such problems fall within the scope of zeroth-order, or gradient-free,
optimization.
Particle swarm optimization (PSO) \cite{kennedy1995particle} is a family
of metaheuristic algorithms based on a system of $n$ interacting particles
and is among the most popular zeroth-order methods used by practitioners.
Other notable methods include simulated annealing
\cite{kirkpatrick1983optimization}, Bayesian optimization
\cite{frazier2018tutorial}, genetic algorithms
\cite{holland1992adaptation}, evolution strategies
\cite{hansen2016cma} and finite-difference methods
\cite{conn2009introduction}.

The algorithm we consider is \textit{momentum-free}%
\footnote{The usual PSO method also includes a momentum term,
which adds a fraction of the previous displacement to the update;
for simplicity, we do not consider this term here.} PSO.
In this variant, a system of $n$ particles
$(x^1_k,\dots,x^n_k)_{k\geq0}$ evolves according to the recursion
\begin{equation}
  \label{eq:algo_intro}
x_{k+1}^i
= x_k^i
+ \gamma \beta \bigl(\theta_k-x_k^i\bigr)
+ \gamma \bar\beta \bigl(y_k^i-x_k^i\bigr)
+ \sqrt{\gamma}\sigma\,\xi_{k+1}^i,
\qquad i\in[n],
\end{equation}
where $\gamma>0$ is the step size and
$\beta,\bar\beta,\sigma\geq0$, with $\beta+\bar\beta=1$.
The drift is a convex combination of displacements toward
the \textit{personal best} $y_k^i$ and the \textit{global best}
$\theta_k$, chosen such that
\begin{align*}
y_k^i
&\in \argmin \left\{
f(x): x\in\{x_0^i,\dots,x_k^i\}
\right\},
\qquad i\in[n],\\
\theta_k
&\in \argmin \left\{
f(y): y\in\{y_k^1,\dots,y_k^n\}
\right\}.
\end{align*}
In case of a tie, $y_k^i$ is the position at the latest iteration
at which particle $i$ attained its best objective value so far.
An arbitrary tie-breaking rule is used to define $\theta_k$.
See Section~\ref{s:PSO} for the formal definitions.
The term $\sqrt{\gamma}\sigma\,\xi_{k+1}^i$ is a random perturbation
used to promote exploration of the space.

While PSO is popular among practitioners, the precise mechanisms
behind its success are still not fully understood. To this end, continuous-time models are helpful as they allow to use tools from dynamical systems and stochastic calculus to analyze the properties of the method in the small-step limit, such as conditions of convergence to the global minimizer, as well as convergence rates. 

However, deriving a continuous-time model of PSO as a small-step limit
is complicated by the discontinuous dependence of the personal- and global-best selectors on the particle trajectories. Indeed, small perturbations
of the trajectories can lead to large changes in the selected
positions.
This prevents a direct application of standard diffusion
approximation results.

To bypass these irregularities, a line of research studies regularized
variants of particle-based optimization dynamics.
An example is provided by Consensus-Based Optimization (CBO)
\cite{pinnau2017consensus},
which can be interpreted as a simplified version of PSO. In CBO, the
particles are driven toward a weighted barycenter of the current swarm,
where the weights favor particles with smaller values of the objective
function. This weighted average is a smooth approximation of the \textit{instantaneous} global best (not the historical one), with a smoothing based on Laplace's principle.

This regularization has proved extremely useful from an analytical point
of view. It makes it possible to derive mean-field limits and to study the
long-time behavior of the resulting dynamics, typically in regimes where
the time step tends to zero, the number of particles tends to infinity,
and the time horizon becomes large. This approach has been extensively
studied; see \cite{carrillo2018analytical,huang2022mean,
fornasier2024consensus,bianchi2025consensus}.

The smoothing approach, however, has several drawbacks. First, it only
provides an approximate model of the original PSO algorithm, which is still
widely used in practice and relies on the actual personal and global bests.
Second, replacing the best particles by linear combinations may alter the
effective optimization target. In the CBO model, when the noise variance
$\sigma^2$ decreases, it is known that, under suitable assumptions on the
initialization and on the objective function, the particles may concentrate
around a consensus point
\cite{carrillo2018analytical,fornasier2025consensus,bianchi2025consensus}.
However, this limiting consensus point need not coincide with a global
minimizer of $f$.

This illustrates a general trade-off. The regularization that makes the
dynamics analytically tractable may also modify the limiting behavior of
the algorithm. The resulting
model is therefore easier to analyze, but it is no longer exactly the PSO
algorithm based on the true personal and global bests. Our goal is instead to derive a continuous-time limit that is more faithful to the original method.

\subsection*{Contributions}
We derive continuous-time dynamics describing the small-step
behavior of the PSO algorithm~\eqref{eq:algo_intro}, with the
number of particles fixed.
We prove tightness of the interpolated particle and personal-best
processes and characterize their weak limit points as
$\gamma\downarrow0$.

Our analysis proceeds in two steps. Under weaker assumptions, we first show that every limit point satisfies relaxed personal-best and global-best conditions. Under additional regularity and nondegeneracy assumptions, we then show that the original selectors are recovered. We present this latter result first.

\paragraph{Main result}
Our main result assumes Gaussian noise with $\sigma>0$,
together with suitable regularity and nondegeneracy conditions
on $f$ and the initialization.
Under these assumptions, Theorem~\ref{th:continuous-time-limit}
shows that every weak limit point is the law of processes
$(\mathbf X,\mathbf Y)$ satisfying
\BEQ\label{eq:continuous-time-dynamics-intro}
dX_t^i
=
\left(\beta\Theta_t+\bar\beta Y_t^i-X_t^i\right)dt
+\sigma\,dB_t^i,
\qquad i\in[n],
\EEQ
where $B^1,\dots,B^n$ are independent standard Brownian motions
and the personal-best processes are given by
\BEQ\label{eq:personal-best-intro}
Y_t^i=X^i_{\tau_t^i},
\qquad
\tau_t^i:=\max\argmin_{0\leq s\leq t}f(X_s^i),
\qquad i\in[n].
\EEQ
Thus, $Y_t^i$ is the position at the latest time when particle
$i$ attained its best objective value so far.
Although the particle trajectories are continuous, the
personal-best processes may have jumps; they are nevertheless
càdlàg (right-continuous with left limits).
The global-best process satisfies
\BEQ\label{eq:global-best-intro}
\Theta_t\in
\argmin\left\{f(y):y\in\{Y_t^1,\dots,Y_t^n\}\right\}
\quad\text{for almost every }t,\quad\text{a.s.}
\EEQ
Moreover, almost surely, the particle attaining the global best
is unique for Lebesgue-almost every $t$.
Consequently, the choice of a tie-breaking rule does not affect
the drift.
The identification of both selectors relies on Brownian
fluctuations and the nondegeneracy assumptions.

\paragraph{Relaxed dynamics}
Theorem~\ref{th:sdi}is the first step in the proof of the main result, and also applies to the noiseless case $\sigma=0$.
It involves two distinct relaxations, associated with the
personal-best memory and the global-best interaction.

The first concerns the lack of closedness of the graph of
the personal-best selector.
Even if particle trajectories converge, their latest-record
positions need not converge to the latest records of the
limiting trajectories.
We introduce a set-valued extension of the personal-best
map with a closed graph.
For a continuous trajectory $x$, let $\mathcal R_f(x)$ be
the set of càdlàg paths $y$ such that $y_0=x_0$ and
\[
\forall\,0\leq s<t,\qquad
y_t\in\argmin\left\{
f(u):u\in\{y_s\}\cup x([s,t])
\right\}.
\]
This condition requires both that $y_t$ be a record position
and that the stored records remain consistent over time:
between two times $s<t$, the memory can only retain $y_s$
or select a minimizing position visited during $[s,t]$.
It therefore rules out arbitrary changes between old
record positions.
We prove that this correspondence has a closed graph
containing the graph of the latest-record selector,
and that every weak limit satisfies
$Y^i\in\mathcal R_f(X^i)$.
The limiting memory need not, however, select the latest
record.

The second relaxation concerns the global best.
Rapid switching between particles leads to the convexified
constraint
\[
\Theta_t\in
\operatorname{conv}
\left\{
Y_t^i:
i\in\argmin_{j\in[n]}f(Y_t^j)
\right\}
\quad\text{for almost every }t,\quad\text{a.s.}
\]
Together with~\eqref{eq:continuous-time-dynamics-intro},
these two conditions characterize the weak limit points
through a stochastic differential inclusion.

To pass from Theorem~\ref{th:sdi} to Theorem~\ref{th:continuous-time-limit}, it remains to show that the relaxed conditions reduce to the original rules. This is obtained thanks to the additional regularity assumptions of Theorem~\ref{th:continuous-time-limit}, and the presence of Brownian noise, which allows to remove ambiguities almost surely, and for almost every $t$.

An example in Section~\ref{ss:convexification} shows that
convexification is necessary in the noiseless case:
rapid switching of the global-best particle can produce
a limiting interaction given by a convex combination
of tied record positions.

\subsection*{Related work}

\paragraph{Early mathematical analyses of PSO}
Early analyses of PSO focused on the stability of particle trajectories
and on the influence of the algorithm parameters.
Clerc and Kennedy~\cite{clerc2002particle} and
Trelea~\cite{trelea2003particle} study simplified deterministic dynamics
and identify parameter regions associated with convergence or instability.
Poli~\cite{poli2009mean} derives recursions for the first and second
moments while retaining the random coefficients, under the assumption
that the personal and global bests remain fixed.
Fern\'andez-Marti\'nez and Garc\'ia-Gonzalo~\cite{fernandez2010stochastic}
study stochastic stability for linear continuous and discrete PSO models
with a random center of oscillation.
Continuous-time formulations also predate CBO:
Emara and Abdel Fattah~\cite{emara2004continuous} propose deterministic
swarm dynamics with a differential equation approximating the evolution
of the personal best, and establish Lyapunov stability results.
For a broader overview of PSO variants and applications,
see~\cite{shami2022particle}.

A closer connection with the present work is provided by
Yuan and Yin~\cite{yuan2015analyzing}, who use stochastic approximation
and weak convergence methods to analyze a class of PSO algorithms.
They obtain an ordinary differential equation limit and a diffusion
limit for centered and rescaled fluctuations.
Their assumptions impose continuity and averaging conditions on
the functions representing the personal and global bests.
Our analysis instead addresses the discontinuous record selectors
defined by the particle trajectories themselves, with Brownian noise
remaining in the small-step limit.
 
\paragraph{CBO and models with memory}
Pinnau et al.~\cite{pinnau2017consensus}
introduce CBO and its mean-field formulation.
Carrillo et al.~\cite{carrillo2018analytical}
develop an analytical framework for well-posedness and consensus
formation, and relate the consensus point to global minimizers through
the Laplace principle.
Huang and Qiu~\cite{huang2022mean} prove a mean-field limit,
while Fornasier et al.~\cite{fornasier2024consensus}
establish global optimization guarantees using a quantitative Laplace
principle and estimates for the particle approximation.
Consensus formation alone does not ensure convergence to an exact
global minimizer for fixed parameters; the optimization error must
also be controlled.
Bianchi et al.~\cite{bianchi2025consensus}
study the long-time behavior of a CBO variant with small but fixed
noise intensity, and provides long run convergence guarantees at both the mean-field and
finite-particle levels.

Several works bring these models closer to PSO by incorporating memory
or inertia.
Totzeck and Wolfram~\cite{totzeck2020personal} propose CBO variants
with either exact personal-best records or weighted averages of past
positions; their well-posedness analysis concerns the weighted-memory
model.
Grassi and Pareschi~\cite{grassi2021particle} formulate stochastic
PSO dynamics with a regularized global best and a differential equation
for the memory variable, and formally derive mean-field and
small-inertia limits.
Cipriani et al.~\cite{cipriani2022zero} rigorously establish
the zero-inertia limit linking these mean-field PSO models to CBO.
Huang et al.~\cite{huang2023global} analyze consensus formation
and proximity to global minimizers for models with regularized memory;
their quantitative mean-field approximation is proved for the case
without memory.
The limit studied here keeps the number of particles fixed and starts
from the discrete personal-best and global-best rules
in~\eqref{eq:algo_intro}, without regularizing either selector.

\paragraph{Differential inclusions, convexification, and record selectors}
Differential inclusions are a standard tool for analyzing stochastic
algorithms with discontinuous or set-valued drifts.
Bena\"im et al.~\cite{benaim2005stochastic}
extend the dynamical systems approach to stochastic approximation
to differential inclusions.
Nguyen and Yin~\cite{nguyen2023stochastic} obtain differential
inclusion limits and stochastic differential inclusion limits for
centered and rescaled iterates.
The convexification caused by rapid switching is also familiar from
relaxed control theory; see Artstein~\cite{artstein1978relaxed}.
In our setting, it describes the cumulative effect of switching
between particles with nearly equal record values.

The personal-best process raises a separate issue: selecting the
latest minimizer is not a continuous operation on trajectories.
Related difficulties for nonunique maximizers are studied by
Ferger~\cite{ferger2004continuous} and
Seijo and Sen~\cite{seijo2011continuous} in extensions of the
argmax continuous mapping theorem.
Here, the limit must also preserve consistency between records
stored at different times, which leads to the relaxed record
condition in Theorem~\ref{th:sdi}.
Finally, our recovery of the exact selectors uses Girsanov's theorem
and properties of Brownian record times.
The use of a change of measure to remove drift when studying
collisions and avoidance of singular sets has precedents in
Ichiba and Karatzas~\cite{ichiba2010collisions} and
Sarantsev~\cite{sarantsev2015triple}.
We use this approach to identify the latest personal-best records
and to prove that ties for the global best occur only on a set
of times of Lebesgue measure zero.

\subsection*{Outline} 
The paper is organized as follows. Section~\ref{s:PSO} introduces the discrete PSO algorithm and the main assumptions,. Section~\ref{s:continuous-time} introduces the continuous-time limit as a stochastic differential inclusion and discusses the noiseless regime. Section~\ref{s:proof} proves the stochastic differential inclusion result, while Section~\ref{s:girsanov} shows that, under nondegenerate noise, the exact personal- and global-best selectors are recovered. Technical results are collected in Section~\ref{s:technical}.

\subsection*{Notation}
For twice differentiable $f:\bR^d\to\bR$, we denote its gradient
and Hessian by $\nabla f$ and $\nabla^2 f$, and set
$\crit f\eqdef\{x:\nabla f(x)=0\}$.
For $p\geq1$, let $\cC_p\eqdef C(\bR_+;\bR^p)$ and
$\cD_p\eqdef D(\bR_+;\bR^p)$, endowed respectively with locally
uniform convergence and the Skorokhod $J_1$ topology.
On $[0,T]$, a metric inducing $J_1$ is
\[
d_{J_1,T}(x,y)
\eqdef\inf_{\lambda\in\Lambda_T}
\left(
\|\lambda-\mathrm{id}\|_{\infty,T}
\vee\|x-y\circ\lambda\|_{\infty,T}
\right),
\]
where $\Lambda_T$ is the set of increasing homeomorphisms of
$[0,T]$ onto itself and
$\|z\|_{\infty,T}\eqdef\sup_{t\leq T}\|z_t\|$.
Convergence $x_n\to x$ in $\cD_p$ means
$d_{J_1,T}(x_n,x)\to0$ for every continuity point $T>0$ of $x$.
Both spaces carry their Borel $\sigma$-fields.

\section{The PSO Algorithm}\label{s:PSO}

Let $f:\bR^d\to \bR$ be a function.
Let $n\geq 1$ be the number of particles. On some probability space, and for each $\gamma>0$,
we introduce the random vector $\bs x_0^\gamma = (x_0^{1,\gamma},\dots,x_0^{n,\gamma})$  on $(\bR^{d})^n$, representing the initial particles positions.
We also consider the process $(\bs \xi_k)_{k\geq 1}$, given by $\bs \xi_k= (\xi_k^{1},\dots,\xi_k^{n})$ on~$(\bR^{d})^n$, and a noise parameter $\sigma \geq 0$. 
\begin{assumption} 
\label{hyp:stab} The following statements hold.
\begin{enumerate}[label=\roman*)]
\item \label{hyp:girsanov-free} The function $f$, defined on $\bR^d\to\bR$, 
is twice continuously differentiable and coercive;
\item $d\geq 2$ and $f$ has isolated critical points.
\end{enumerate}
\end{assumption}

\begin{assumption}
\label{hyp:random} The following conditions hold.
\begin{enumerate}[label=\roman*)]
\item \label{hyp:girsanov-free-1} The random variables $(\bs x_0^\gamma)_{\gamma>0}$ have the same distribution $\mu_0$, which is compactly supported;
\item \label{hyp:girsanov-free-2} $(\bs \xi_k)_{k\geq 1}$ are iid standard Gaussian vectors in $(\bR^{d})^n$, independent from $(\bs x_0^\gamma)_{\gamma>0}$ ;
\item $\mu_0$ is absolutely continuous w.r.t. Lebesgue measure on $\bR^{dn}$ ;
\item $\sigma>0$.
\end{enumerate}
\end{assumption}
The coercivity of $f$ and the uniform compact support assumption on the
initial particles ensure that
the personal bests and global bests remain in a deterministic compact
set. The smoothness assumption on $f$, along with the condition on the isolated critical
points and the law of the initial particles, are used to rule out pathological behaviors of the record times,
and to guarantee the uniqueness of the global best (see Prop.~\ref{prop:unique-best-discrete}).
The Gaussian assumption on $\xi_k$ is mainly for convenience.

Let $\beta\in [0,1]$ be fixed, and set $\bar \beta\eqdef 1-\beta$.   
For every $\gamma>0$, starting from the initial position $\bs x_0^\gamma$, we recursively
define the processes $(\bs x_k^\gamma)_{k\geq0}$ and
$(\bs y_k^\gamma)_{k\geq0}$, taking values in $\bR^{dn}$, together
with the process $(\theta_k^\gamma)_{k\geq0}$, taking values in
$\bR^d$, through the iterations:
\begin{equation}
  \label{eq:algo}
  \begin{aligned}
&y_k^i
=
x_{\kappa_k^i}^i,
\qquad
\kappa_k^i
\eqdef
\max \arg\min_{0\leq l\leq k} f(x_l^i), \qquad i\in [n], \\
& \theta_k = y_k^{i_k^*},\qquad i^*_k = \min \arg\min_{j\in[n]} f(y_k^j) , \\
&x_{k+1}^i =   x_k^{i}   + \gamma \beta \bigl(\theta_k - x_k^i\bigr) + \gamma \bar \beta \bigl( y_k^{i} - x_k^i\bigr)
    + \sqrt{\gamma} \sigma\,\xi_{k+1}^{i}\,,    \qquad i\in [n].
  \end{aligned}
\end{equation}
Here,  $y_k^{i}$ represents the position of the personal best of a particle $i\in [n]$.
It is defined as the \textit{last} position of the particle $i$, at which this particle attained its personal record.
The random integer~$i^*_k$ represents the index of the global best particle at iteration $k$.
The vector $\theta_k = y_k^{i_k^*}$ is the global best, defined as the element within the set of personal bests positions $y_k^1,\dots, y_k^n$,
which achieves the best value of $f$.
In case of a tie, the particle with the smallest index is selected.
However, Prop.~\ref{prop:unique-best-discrete}
reveals that this tie-breaking rule is generally irrelevant when $\sigma > 0$, as the particle achieving the global record is almost always unique thanks to the Brownian noise.
 \begin{prop} \label{prop:unique-best-discrete}
   Under Assumptions~\ref{hyp:stab},~\ref{hyp:random}, $f(y_k^j)\neq f(y_k^i)$ almost surely for every $k\geq 0$, $i\neq j$.
   In particular, $\arg\min_{j\in[n]} f(y_k^j)$ is a singleton, with probability one.
\end{prop}
\begin{proof}[Proof sketch]
  Let $k\geq 0$ and $i\neq j$. The event $f(y^i_k) = f(y^j_k)$ is included in the union of events $\{f(x_l^i) = f(x_m^j)\}$ over $l,m\leq k$.
  We can show that the random vector $(x^i_m, x^j_l)$ has a density on $\bR^{2d}$, while the set $\{(x,x^\prime)\in \bR^{2d}:f(x) = f(x^\prime)\}$
  has zero Lebesgue measure. Hence, the event $f(y^i_k) = f(y^j_k)$ has zero probability.
  For the complete proof see Section~\ref{s:technical}.
\end{proof}

\section{Continuous-time dynamics}
\label{s:continuous-time}
\subsection{Convergence to weak solutions}

We now consider trajectories in continuous-time. 
For a continuous function $x\in \cC_d$ and $t\geq0$, define
$\tau_t(x)
\eqdef
\max \arg\min_{0\leq s\leq t} f(x_s)$,
as the last time before $t$ when $x$ reached its best record.
The map
\begin{equation}
  \label{eq:pbest}
  \begin{aligned}
  \pbest: \cC_d &\to\cD_d  \\
  x &\mapsto (x_{\tau_t(x)})_{t\geq 0}
\end{aligned}
\end{equation}
is the continuous-time personal best selector,
which to a given continuous trajectory $x$, associates the positions of last record.
\begin{lemma}\label{lem:pbest-well-defined}
  The map $\pbest$ is well defined on $\cC_d\to\cD_d$, measurable, and non-anticipative.
\end{lemma}
\begin{proof}
  See Section~\ref{s:technical}.
\end{proof}
For $\bs x \in \cC_{dn}$, we use the notation $\pbest(\bs x)\eqdef (\pbest(x^1),\dots,\pbest(x^n))$ to designate the componentwise application of $\pbest$.

We also define the map $J : \bR^{dn}\to [n]$ as
$J(\bs y) = \min \arg\min_{i\in [n]}f(y^i)$ for every $\bs y=(y^1,\dots, y^n)$.
The map
\begin{eqnarray*}
  \gbest: \bR^{dn} &\to&\bR^d \nonumber \\
  \bs y &\mapsto& y^{J(\bs y)} \nonumber
\end{eqnarray*}
is the global best selector. It is a Borel map.
Note however that the function $t\mapsto \gbest(\bs y_t)$ is not in general a càdlàg function, even if $\bs y$ is càdlàg.
As in the discrete scenario, the definition of $\gbest$ uses an arbitrary tie-breaking rule, in the case where there are several best particles at a given instant. As will be clarified in Th.~\ref{th:continuous-time-limit}, this tie-breaking rule is in fact irrelevant when $\sigma > 0$, as there is almost always a unique particle achieving the global best.

Consider the following system of equations, which describe the continuous-time PSO dynamics.
\begin{equation}
\label{eq:continuous-time}
\begin{aligned}
Y_t^i
&=
\pbest(X^i)_t
=
X^i_{\tau_t(X^i)},
\qquad i\in[n],
\\
dX_t^i
&=
\left(
\beta\,\gbest(\bs Y_t)
+\bar\beta Y_t^i-X_t^i
\right)dt
+\sigma\,dB_t^i,
\qquad i\in[n].
\end{aligned}
\end{equation}

\begin{definition}[Weak solution]
\label{def:weak-solution}
Let $\mu_0\in\cP(\bR^{dn})$.
A weak solution of~\eqref{eq:continuous-time} with initial
distribution $\mu_0$ and parameters $(n,\sigma,\beta)$ consists of
\begin{itemize}
\item a filtered probability space
$(\Omega,\cF,(\cF_t)_{t\geq0},\bP)$;

\item an $(\cF_t)_{t\geq0}$-adapted process
$\bs X$ with paths in $\cC_{dn}$;

\item a standard $dn$-dimensional
$(\cF_t)_{t\geq0}$-Brownian motion $\bs B = (B^1,\dots, B^n)$;
\end{itemize}
such that $\bs X_0$ has the distribution $\mu_0$, and, setting
$\bs Y=\pbest(\bs X)$,
one has, almost surely, for every $t\geq0$ and every $i\in[n]$,
\[
X_t^i
=
X_0^i
+
\int_0^t
\left(
\beta\,\gbest(\bs Y_s)
+\bar\beta Y_s^i-X_s^i
\right)\,ds
+
\sigma B_t^i.
\]
\end{definition}
Our goal is to show that the law of such a solution arises as a weak limit of the laws of interpolated discrete-time PSO processes as the step size $\gamma$ tends to zero. To that end, it is convenient to introduce the concept of weak solution law.
\begin{definition}[Weak solution law]
Let $\mu_0\in\cP(\bR^{dn})$. A probability measure $\nu\in\mathcal P(\mathcal C_{dn}\times\mathcal (\cD_d)^n)$
is called a weak solution law of~\eqref{eq:continuous-time} with initial distribution $\mu_0$,
if there exists a weak solution in the sense of Def.~\ref{def:weak-solution} such that
$(\mathbf X,\pbest(\bs X))$ has the law $\nu$.
\end{definition}
Existence of a weak solution with initial distribution
$\mu_0$ could be proven using Girsanov's theorem \cite[Th.~3.5.1]{shreve1991brownian}, together with~\cite[Lem.~2.4]{ren2020comparison}.
The main issue here, however, is not the abstract existence of such a
solution, but the fact that it correctly describes the small-step-size
behavior of the discrete algorithm, as established by the weak
convergence result below.

\paragraph{Continuous-time embeddings.} We embed the discrete-time algorithm into a continuous-time
framework. For every $\gamma>0$, we define the continuous-time interpolation of the
particle system defined in \eqref{eq:algo} by
\begin{equation}
X_t^{i,\gamma}
=
x_k^{i,\gamma}
+
\left(\frac{t}{\gamma}-k\right)
\bigl(x_{k+1}^{i,\gamma}-x_k^{i,\gamma}\bigr),
\qquad
t\in[k\gamma,(k+1)\gamma),
\label{eq:Xgamma}
\end{equation}
and the piecewise-constant interpolation of the personal best process by
\begin{equation}\label{eq:Ygamma}
Y_t^{i,\gamma}
=
y_k^{i,\gamma},
\qquad
t\in[k\gamma,(k+1)\gamma).
\end{equation}
We write $\mathbf X_t^\gamma
=
(X_t^{1,\gamma},\dots,X_t^{n,\gamma}),
$ and $\bs Y_t^\gamma
=
(Y_t^{1,\gamma},\dots,Y_t^{n,\gamma})$.
Thus, $\mathbf X^\gamma$ is a random variable with values in $\cC_{dn}$,
whereas $\bs Y^\gamma$ is a random variable with values in $(\cD_d)^n$. We now state our main result.

\begin{theorem}[Continuous-time limit]
\label{th:continuous-time-limit}
Under Assumptions~\ref{hyp:stab},~\ref{hyp:random}, the family of r.v.
$\bigl(\mathbf X^\gamma,\bs Y^\gamma\bigr)_{\gamma>0}$ is tight on $\cC_{dn}\times(\cD_d)^n$,
and each limit point as $\gamma \downarrow 0$ is a weak solution law to (\ref{eq:continuous-time}).
\end{theorem}

\subsection{Sketch of the proof}

Let $\bs X^\gamma$ and $\bs Y^\gamma$ denote the interpolated processes
defined in Eqs.~\eqref{eq:Xgamma} and~\eqref{eq:Ygamma}. Let us also introduce the piecewise-constant embedding of the discrete global best $\theta_k^\gamma$ as
\begin{equation}\label{eq:Thetagamma}
\Theta_t^\gamma=\theta_k^\gamma,
\qquad
t\in[k\gamma,(k+1)\gamma).
\end{equation}
Let us first point out a difficulty. A natural approach would consist
in proving the tightness of the triplet
$(\bs X^\gamma,\bs Y^\gamma,\Theta^\gamma)$ as $\gamma \downarrow 0$ on
$\cC_{dn}\times(\cD_d)^n\times\cD_d$, and then showing that its limit
points have the form $(\bs X,\bs Y,\Theta)$, where
$\bs Y=\pbest(\bs X)$ and
$\Theta_t=\gbest(\bs Y_t)$. This approach fails. On the one hand,
$(\bs X^\gamma,\bs Y^\gamma)$ converges in law along subsequences as
$\gamma\downarrow0$ to a limit $(\bs X,\bs Y)$; see
Prop.~\ref{prop:tightness-gamma}. On the other hand, the family
$(\Theta^\gamma)_\gamma$ is not tight in general, as the global best position could switch infinitely often between several particles with close objective value.
Instead, we consider the integral process
\begin{equation}\label{eq:def_Atgamma}
A_t^\gamma\eqdef\int_0^t\Theta_s^\gamma\,ds,
\end{equation}
which is a random element of $\cC_d$. We show that the family $(A^\gamma)_{\gamma >0}$
is tight as $\gamma\downarrow0$, and therefore converges in law along
subsequences to a process $A$ whose sample paths are absolutely
continuous. We then define the limiting global-best $\Theta$
as an almost-everywhere derivative of $A$.

Thus, unlike $\bs X$ and $\bs Y$, which are obtained as weak limits of
$\bs X^\gamma$ and $\bs Y^\gamma$, respectively, $\Theta$ is not a
weak limit of $\Theta^\gamma$. It is only through their respective
integrals that the discrete and limiting global-best processes are
related.

To understand the consequence of this fact, define, for
$\bs y=(y^1,\dots,y^n)$,
\begin{equation}
F(\bs y)
\eqdef
\left\{
y^i:\ i\in\arg\min_{j\in[n]}f(y^j)
\right\}.\label{eq:F}
\end{equation}
At the discrete level,
$\Theta_t^\gamma\in F(\bs Y_t^\gamma)$ for every $t$, but this
pointwise inclusion cannot be passed directly to the limit. Indeed,
difference quotients of $A$ arise as limits of local averages
\[
\frac{1}{2h}\int_{t-h}^{t+h}\Theta_s^\gamma\,ds,
\]
which may combine several values of $F(\bs Y_s^\gamma)$ and therefore
naturally belong to their convex hull. This is established in
Prop.~\ref{prop:limit-global-best}, which shows that
$\Theta_t\in\conv F(\bs Y_t)$ for Lebesgue-almost every $t>0$.

A second difficulty concerns the identification of $\bs Y$ from
$\bs X$. At the discrete level, $Y^{i,\gamma}$ is the personal best of
$X^{i,\gamma}$ observed at the grid points. This property cannot be
passed directly to the limit, notably because the selector $\pbest$ is
not continuous from $\cC_d$ to $\cD_d$.
Weak convergence therefore only guarantees that $(X^i,Y^i)$ belongs
almost surely to the closure of the graph of $\pbest$. Rather than
characterizing this closure directly, we identify a sufficiently
sharp closed enlargement of the graph. It retains two properties of
the personal-best process: $Y_t^i$ must be a record point of the
  trajectory of $X^i$, and the stored record must be temporally
consistent. More precisely, between two times $s<t$, it can only be
retained at $Y_s^i$ or replaced by a minimizing point visited by
$X^i$ during $(s,t]$. This prevents ``jumps'' to arbitrary old records. We provide this result in Prop.~\ref{prop:limit-personal-best}.
Putting all pieces together, we obtain the following theorem.
\begin{theorem}[Stochastic differential inclusion]
\label{th:sdi}
Under Assumption~\ref{hyp:stab}\ref{hyp:girsanov-free} and
Assumption~\ref{hyp:random}\ref{hyp:girsanov-free-1}--%
\ref{hyp:girsanov-free-2}, the family
$\bigl(\bs X^\gamma,\bs Y^\gamma\bigr)_{\gamma>0}$
is tight on $\cC_{dn}\times(\cD_d)^n$.
Let $\nu$ be any weak limit point of this family as $\gamma \downarrow 0$. There exist a
filtered probability space
$(\Omega,\cF,(\cF_t)_{t\geq0},\bP)$, an adapted process
$(\bs X,\bs Y)$ with law $\nu$, a progressively measurable process
$\Theta$ with values in $\bR^d$, and a standard $dn$-dimensional
$(\cF_t)$-Brownian motion
$\bs B=(B^1,\dots,B^n)$ such that the following properties hold.
\begin{enumerate}[label=\roman*)]
\item $\operatorname{Law}(\bs X_0)=\mu_0$, $\bs Y_0=\bs X_0$ a.s;
\item almost surely, for every $t\geq0$ and every $i\in[n]$,
\[
X_t^i
=
X_0^i
+
\int_0^t
\left(
\beta\Theta_s+\bar\beta Y_s^i-X_s^i
\right)\,ds
+
\sigma B_t^i,
\qquad t\geq0,
\]
\item $\Theta_t\in\conv F(\bs Y_t)$, $\bP\otimes dt\text{-a.e.}$;
\item with probability one, for every $i\in[n]$,
\[
\forall t\geq0,\ \forall s\in[0,t),\qquad
Y_t^i
\in
\arg\min
\left\{
f(u):
u\in\{Y_s^i\}\cup X^i([s,t])
\right\}.
\]
\end{enumerate}
\end{theorem}
 Th.~\ref{th:sdi} is an intermediate result which we use to prove Th.\ref{th:continuous-time-limit}, but it is also of independent interest, notably because it allows to handle the noiseless case $\sigma =0$.
Formally, Th.~\ref{th:sdi} states the following differential inclusion:
\begin{equation}\label{eq:sdi_informal}
dX^i_t \in (\beta \conv F(\bs Y_t) + \bar \beta Y^i_t - X^i_t)dt + \sigma dB^i_t\,.
\end{equation}
 The result provides only a relaxed characterization of the
limit: the process $Y^i$ satisfies a closed record-consistency
property, while the global-best process is characterized through the
inclusion $\Theta_t\in\conv F(\bs Y_t)$. 
To recover the original
personal- and global-best selectors and prove Th.~\ref{th:continuous-time-limit}, we use the remaining assumptions:
$\sigma>0$, $d\geq2$, isolated critical points of $f$, and
$\mu_0\ll\operatorname{Leb}$.

The argument relies on Girsanov's theorem, which
makes the law of each particle trajectory equivalent to
that of a nondegenerate Brownian motion.
Since a Brownian motion in dimension $d\geq2$
almost surely does not hit any prescribed point, and since the set of
critical points of $f$ is countable, the particle trajectories almost
surely never visit $\crit f$. Around points with nonzero gradients, the Brownian noise produces fluctuations which allow every record level to be improved immediately. This removes the ambiguity left by the closed
record-consistency property and yields
\[
Y^i=\pbest(X^i),\qquad i\in[n].
\]
This point is established in Prop.~\ref{prop:pbest}.
The same crossing argument shows that the global best record is
attained by a unique particle for $\bP\otimes dt$-almost every $(\omega,t)$. 
Consequently,
$F(\bs Y_t)$ is almost everywhere a singleton, so that
\[
\Theta_t=\gbest(\bs Y_t)
\qquad \bP\otimes dt\text{-a.e.}
\]
This point is established in Cor.~\ref{cor:gbest}.
Combining these two selector-recovery results with Th.~\ref{th:sdi} shows that every weak limit point is a weak solution law of \eqref{eq:continuous-time}, which proves Th.~\ref{th:continuous-time-limit}.


\subsection{The role of convexification in the noiseless case}\label{ss:convexification}

We argue that, when $\sigma =0$, the convexification appearing in Th.~\ref{th:sdi} is not merely a
technical artifact. When $\sigma>0$, the additional assumptions of
Th.~\ref{th:continuous-time-limit} imply that $F(\bs Y_t)$ is a
singleton for $\bP\otimes dt$-almost every $(\omega,t)$, and the
convexification disappears. This need not be the case when
$\sigma=0$.

Indeed, Th.~\ref{th:sdi} also applies in the noiseless setting. It
shows that every weak limit point admits a realization satisfying
\begin{equation}
\label{eq:noiseless-sdi}
\frac{dX_t^i}{dt}
=
\beta\Theta_t+\bar\beta Y_t^i-X_t^i,
\qquad
\Theta_t\in\conv F(\bs Y_t),
\end{equation}
for Lebesgue-almost every $t$, together with the record-consistency
property
\begin{equation}\label{eq:record_consistency_recall}
Y_t^i
\in
\arg\min\left\{
f(u):
u\in\{Y_s^i\}\cup X^i([s,t])
\right\},
\qquad 0\leq s<t.
\end{equation}
The theorem does not assert in this case that
$Y^i=\pbest(X^i)$. Nevertheless, the following simple example shows that the
convex hull in~\eqref{eq:noiseless-sdi} may genuinely be active.

\begin{figure}[h]
  \centering
  \includegraphics[width=0.8\linewidth]{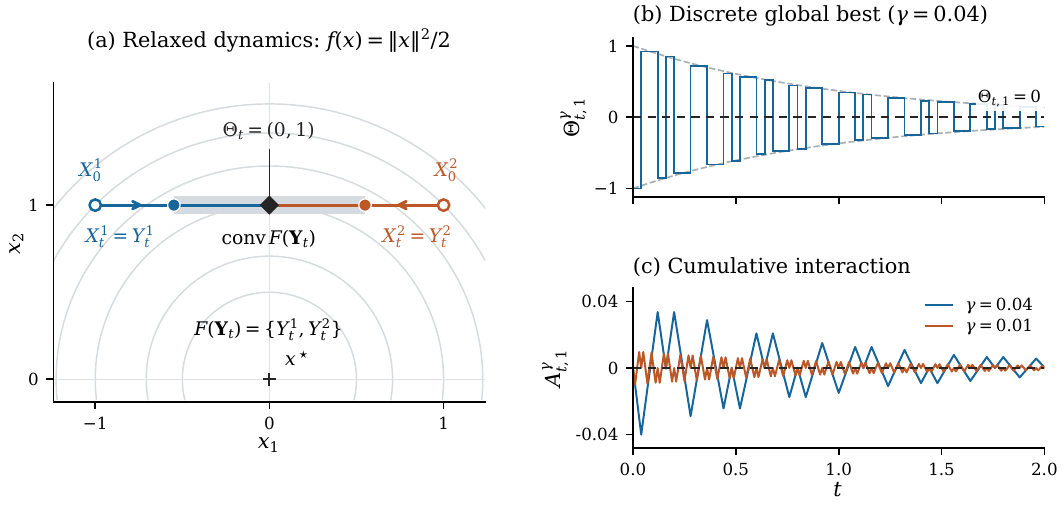}
  \caption{Convexification in noiseless PSO with $n=2$, $\beta=1$,
  $f(x)=\|x\|^2/2$, and initial positions $X_0^1=(-1,1)$ and
  $X_0^2=(1,1)$.
  (a) The limiting trajectories are $X_t^1=(-e^{-t},1)$ and
  $X_t^2=(e^{-t},1)$, with $Y_t^i=X_t^i$; filled circles show
  the positions at $t=0.6$ and grey curves are level sets of $f$.
  The interaction $\Theta_t=(0,1)$ lies strictly between the two
  record positions for every finite $t$.
  (b) First coordinate of the piecewise-constant discrete global best
  for $\gamma=0.04$; the grey dashed curves are $\pm e^{-t}$.
  (c) First coordinate of
  $A^\gamma_t=\int_0^t\Theta^\gamma_s\,ds$
  for two step sizes. The black dashed lines in (b) and (c)
  show the first coordinates of $\Theta_t$ and its integral,
  both equal to zero.}
  \label{fig:noiseless-convexification}
\end{figure}

\begin{example}[Noiseless case, two tied particles and a quadratic objective]
Set $d=n=2$, $\beta=1$, $\sigma = 0$, and $f(x)= \|x\|^2$. Consider the continuous-time relaxed dynamics \eqref{eq:noiseless-sdi} initialized at $X_0^1 = (-1,1)$ and $X_0^2 = (1,1)$, which write
\[
dX_t^i = (\Theta_t - X_t^i) dt, \qquad  \Theta_t \in \operatorname{conv} F(\bs Y_t),
\]
along with the record-consistency property \eqref{eq:record_consistency_recall}.
These dynamics admit a unique solution which writes
\[
\Theta_t = \left( 0,1 \right), \qquad 
X_t^1 = (-e^{-t},1), \qquad X_t^2 = (e^{-t},1) , \qquad F(\bs Y_t) = \{ X_t^1,X_t^2 \}, \qquad \forall t \geq 0.
\]
In particular,
$\Theta_t\notin F(\bs Y_t)$, while
$\Theta_t\in\conv F(\bs Y_t)$.
Hence the convexification in Th.~\ref{th:sdi} is essential in the
noiseless regime. Figure~\ref{fig:noiseless-convexification} illustrates this behavior numerically.
\end{example}
\begin{proof}
First note that by strong convexity of $f$, we have
\[\tfrac{d}{dt} f(X_t^i) = \langle \nabla f(X_t^i), \Theta_t - X_t^i  \rangle  \leq f(\Theta_t) - f(X_t^i) - \frac{1}{2} \|\Theta_t - X_t^i\|^2 \leq 0,\]
where the last inequality comes from the fact that $f(\Theta_t) \leq \min_{j \in [n]} f(Y_t^j) \leq f(X_t^i)$ by convexity and the record consistency properties (Th. \ref{th:sdi} (iii)-(iv)).

Therefore the function $t \mapsto f(X_t^i)$ is nonincreasing. We argue that this implies that $X_t$ is equal to $Y_t$ at all times. Indeed, fix $t \geq 0$. By the record property \eqref{eq:record_consistency_recall} with $s=0$, there exists $u \in [0,t]$ such that $Y_t^i = X_u^i$. From monotonicity and \eqref{eq:record_consistency_recall}, we get
$
f(X_u^i)=f(Y_t^i)=f(X_t^i).
$
Thus $f(X_r^i)$ is constant on $[u,t]$, and the previous inequality
implies $\Theta_r=X_r^i$ for almost every $r\in[u,t]$. Hence
$\frac{d}{dt}X_r^i=0$ on $[u,t]$, so $X_u^i=X_t^i$. We deduce that
\[
Y_t^i=X_t^i,
\qquad t\geq0,\quad i=1,2.
\]

Because the particles lie initially on the convex segment $[-1,1]\times \{1\}$, they remain inside for all $t$. Write the coordinates $X_t^i = (z_t^i,1)$ for $i = 1,2$, and define the quantities
\[ m_t = \frac{ z_t^1 + z_t^2 }{2}, \qquad d_t = z_t^2 - z_t^1.\]
Note that $\dot{d_t} = -d_t$, therefore $d_t = 2e^{-t}$; in particular, $d_t$ is positive. Since
$f(X^2_t) - f(X^1_t) = (z_t^2)^2 - (z_t^1)^2 = 2m_td_t,
$
the global best is determined by the sign of $m_t$. 

If $m_t < 0$, then particle 2 is the unique global best, $F(\bs Y_t) = \{X_t^2\}$, $\Theta_t = X_t^2$ and
\[
  \dot{m_t} = z_t^2 - m_t = \frac{d_t}{2} > 0.
\]
If on the other hand $m_t > 0$, then similarly $F(\bs Y_t) = \{X_t^1\}$, $\Theta_t = X_t^1$ and 
\[
  \dot{m_t} = z_t^1 - m_t = -\frac{d_t}{2} < 0.
\]
It follows that $\frac{d}{dt}m_t^2 = - |m_t| d_t \leq  0$. Since $m_0 = 0$ at initialization, we have $m_t = 0$ for any $t$. We deduce that $ \frac{X_t^1 + X_t^2}{2} = (0,1)$. Noting that 
\[ \frac{d}{dt}\left( \frac{X_t^1 + X_t^2}{2} \right) = \Theta_t - \frac{X_t^1 + X_t^2}{2}, 
\]
we conclude that $\Theta_t = (0,1)$, and the expressions for $X_t^1,X_t^2$ follow.
\end{proof}

\section{Proof of the stochastic differential inclusion theorem (Th.~\ref{th:sdi})}
\label{s:proof}

Assumptions~\ref{hyp:stab}-\ref{hyp:girsanov-free},~\ref{hyp:random}-\ref{hyp:girsanov-free-1} and \ref{hyp:random}-\ref{hyp:girsanov-free-2}
are in force throughout this section.

\subsection{Interpolated processes}

We define $\iota:\cD_d\to\cC_d$ by
\begin{equation}
\label{eq:def-iota}
\iota(x)_t=\int_0^t x_s\,ds.
\end{equation}
We use the same notation $\iota(\bs x) \eqdef (\iota(x^1),\dots,\iota(x^n))$ when $\bs x\in (\cD_d)^n$.
We defined $A^\gamma = \iota(\Theta^\gamma)$, where $\Theta_t^\gamma = \theta_k^\gamma$ for $t\in [k\gamma, (k+1)\gamma)$.
Define also $\bs M^\gamma=(M^{1,\gamma},\dots,M^{n,\gamma})$, where
\[
M_t^{i,\gamma}
\eqdef
\sigma \gamma^{-1/2}\int_0^t
\xi_{\lfloor s/\gamma\rfloor+1}^i\,ds,
\qquad i\in[n].
\]
By Donsker’s theorem, $M^{i,\gamma}$ converges weakly to a Brownian motion as $\gamma \downarrow 0$.
\begin{lemma} \label{lem:decompX} 
Let $\bs X^\gamma, \bs Y^\gamma$ as defined by \eqref{eq:Xgamma}-\eqref{eq:Ygamma}, and $\bs M^\gamma$ as above.
For every $i\in[n]$,
\begin{equation}
\label{eq:decompX}
X^{i,\gamma}
=
X_0^{i,\gamma}
+\beta A^\gamma
+\bar\beta\,\iota(Y^{i,\gamma})
-\iota(X^{i,\gamma})
+M^{i,\gamma}
+r^{i,\gamma},
\end{equation}
where
$r_t^{i,\gamma}
\eqdef
\int_0^t
\left(
X_s^{i,\gamma}
-
X_{\gamma\lfloor s/\gamma\rfloor}^{i,\gamma}
\right)\,ds.
$
\end{lemma}
\begin{proof}
By definition of the piecewise-linear interpolation \eqref{eq:Xgamma} and the discrete dynamics \eqref{eq:algo},
\begin{align*}
X_t^{i,\gamma}
&=
X_0^{i,\gamma}
+
\frac1\gamma
\int_0^t
\left(
x_{\lfloor s/\gamma\rfloor+1}^{i,\gamma}
-
x_{\lfloor s/\gamma\rfloor}^{i,\gamma}
\right)\,ds
\\
&=
X_0^{i,\gamma}
+
\int_0^t
\left(
\beta\theta_{\lfloor s/\gamma\rfloor}^\gamma
+\bar\beta y_{\lfloor s/\gamma\rfloor}^{i,\gamma}
-x_{\lfloor s/\gamma\rfloor}^{i,\gamma}
+\sigma\gamma^{-1/2}\xi_{\lfloor s/\gamma\rfloor+1}^i
\right)\,ds.
\end{align*}
Using the definitions above and
$x_{\lfloor s/\gamma\rfloor}^{i,\gamma}
=
X_{\gamma\lfloor s/\gamma\rfloor}^{i,\gamma}$
gives~\eqref{eq:decompX}.
\end{proof}

The quantity $r^{i,\gamma}$ is the error produced by replacing the stepwise process $X_{\gamma\lfloor s/\gamma\rfloor}^{i,\gamma}$ with its linear interpolation $X_s^{i,\gamma}$. It will be shown to vanish as $\gamma\downarrow0$.

\subsection{Tightness over step size}

Having defined the continuous-time interpolation, the next crucial step is to show that a continuous-time limit
actually exists. We do that by proving the tightness of the family of processes $(\bs X^\gamma, \bs Y^\gamma, A^\gamma)$
as $\gamma\downarrow 0$. 

We begin by establishing a uniform moment bound on the particle trajectories.
\begin{lemma}
\label{lem:moment-X}
We have
\[
\sup_{i\in[n],\,k\ge0,\,\gamma\in(0,1/2)}
\bE\| x_k^{i,\gamma}\|^2<\infty.
\]
\end{lemma}

\begin{proof}
  Since $\supp(\mu_0)\subset B(0,R)^n$ for some $R > 0$, the personal best $ y_k^{i,\gamma}$ belongs a.s. to the sublevel set $\{f\leq \sup f(B(0,R))\}$, which is compact, as $f$ is coercive. Thus there exists $R'>0$ such that
$\| y_k^{i,\gamma}\|,\|\theta_k^\gamma\|\le R'$.
Setting $z_k^{i,\gamma}=\beta\theta_k^\gamma+\bar\beta y_k^{i,\gamma}$, we have $\|z_k^{i,\gamma}\|\le R'$.
Using  $\gamma<1/2$ and the independence of the Gaussian increments,
\begin{align*}
\bE\| x_{k+1}^{i,\gamma}\|^2
&\le
(1-\gamma)^2\bE\| x_k^{i,\gamma}\|^2
+\gamma^2R'^2+\gamma\sigma^2d
+2\gamma R'\,\bE\| x_k^{i,\gamma}\| \\
&\le
\left(1-\frac\gamma2\right)\bE\| x_k^{i,\gamma}\|^2
+\gamma C,
\end{align*}
where $C=2R'^2+\sigma^2d$. Hence
$\bE\| x_{k+1}^{i,\gamma}\|^2\le\bE\| x_k^{i,\gamma}\|^2$
whenever $\bE\| x_k^{i,\gamma}\|^2\ge2C$, and the result follows.
\end{proof}

\begin{prop}\label{prop:tightness-gamma}
  For each $i\in [n]$, the family of processes $(\inX^{i,\gamma},Y^{i, \gamma}, A^\gamma)_{\gamma>0}$ on $\cC_d\times \cD_d \times  \cC_d $
  is tight as $\gamma\downarrow 0$, and $r^{i,\gamma}$ converges in probability to zero in $\cC_d$ as $\gamma\downarrow 0$.
\end{prop}
\begin{proof}
  We use the decomposition given by Lemma \ref{lem:decompX}.
  First, we prove that:
  \begin{itemize}
  \item \underline{$(\iota( Y^{i,\gamma}))_{\gamma>0}$ is tight}.
  As in the proof of Lemma \ref{lem:moment-X}, we note that, by compactness of the support of $\mu_0$ and coercivity of $f$, there exists a deterministic $R'>0$ s.t. $\| y_k^{i,\gamma}\|\leq R'$. The family $(\iota( Y^{i,\gamma}))_{\gamma>0}$ 
    belongs to the set of $R'$-Lipschitz continuous functions, null at the origin. The latter set is relatively compact in $\cC_d$.
  \item \underline{$(A^\gamma) = (\iota(\Theta^{\gamma}))_{\gamma>0}$ is tight}.
    We use the same argument, noticing that $\|\theta_k^{\gamma}\|\leq R'$.
  \item \underline{$(M^{i,\gamma})_{\gamma>0}$ is tight}
  . This is a consequence of Donsker's theorem.
  \item \underline{$(\iota(\inX^{i,\gamma}))_{\gamma>0}$ is tight}
  . This is a consequence of the inequality, for every $0<s<t<T$,
    \begin{align*}
      \bE(\|\iota(\inX^{i,\gamma})_t - \iota(\inX^{i,\gamma})_s\|^2) = \bE\left(\left\| \int_s^t \inX_u^{i,\gamma} du\right\|^2\right)
      \leq (t-s) \int_s^t \bE\left(\left\| \inX_u^{i,\gamma} \right\|^2\right) du \ \leq C (t-s)^2\,,
    \end{align*}
    where $C$ is the supremum of $\bE\left(\left\| \inX_u^{i,\gamma} \right\|^2\right)$ w.r.t $i,u$ and $\gamma<\tfrac 12$, which is finite by Lemma~\ref{lem:moment-X}.
    The conclusion follows from the Kolmogorov-Chentsov criterion \cite[Corollary 14.9]{kallenberg1997foundations}.
  \item \underline{$(r^{i,\gamma})_{\gamma>0}$ converges in probability to zero}. For every $T>0$, 
\[
\sup_{t\leq T}\|r_t^{i,\gamma}\|
\leq
\frac{\gamma}{2}
\sum_{k=0}^{\lfloor T/\gamma\rfloor}
\|x_{k+1}^{i,\gamma}-x_k^{i,\gamma}\|.
\]
  Now, setting $z_k^{i,\gamma} = \beta \theta_k^\gamma + \bar\beta  y_k^{i,\gamma}$, we write 
  $  x_{k+1}^{i,\gamma} =  (1-\gamma) x_{k}^{i,\gamma}+\gamma z_k^{i,\gamma} +\sqrt\gamma \sigma \xi_{k+1}^{i}$, which yields:
  \begin{align*}
     \|   x_{k+1}^{i,\gamma}- x_k^{i,\gamma}\| &\leq \gamma \|z_k^{i,\gamma}\| + \gamma \| x_k^{i,\gamma}\|+\sqrt\gamma \sigma \|\xi_{k+1}^{i}\|\,.
  \end{align*}
  By the standing assumptions, $z_k^{i,\gamma}$ is bounded by a deterministic constant $R'>0$. Moreover, by Lem.~\ref{lem:moment-X}, there exists $C>0$ such that $\bE\| x_k^{i,\gamma}\|\leq C$. We deduce:
\[
\sup_{k\gamma\leq T}
\bE\|x_{k+1}^{i,\gamma}-x_k^{i,\gamma}\|
\leq C(\gamma+\sqrt{\gamma}),
\]
for some $C>0$. Therefore,
\[
\bE\left[\sup_{t\leq T}\|r_t^{i,\gamma}\|\right]
\leq C (T+\gamma)(\gamma+\sqrt{\gamma})
\longrightarrow0.
\]
Hence $r^{i,\gamma}\to0$ in probability in $\cC_d$ by Markov's inequality.

\item \underline{$(X^{i,\gamma})_{\gamma>0}$ is tight}.   This is a consequence of the decomposition~(\ref{eq:decompX}) and the previous points.

\item \underline{$(Y^{i,\gamma}))_{\gamma>0}$ is tight}. Let us introduce some notations.
For $x\in\cD_p$ and for every interval $I\subset [0,\infty)$, set $w(x,I)\eqdef\sup_{s,t\in I}\|x_s-x_t\|$.
For $T,\delta>0$, set
\begin{eqnarray*}
w_T(x,\delta)&\eqdef&\sup_{\substack{s,t\in[0,T]\\|t-s|\leq\delta}}\|x_t-x_s\|, \\
w'_T(x,\delta)&\eqdef&
\inf_{\substack{0=t_0<\cdots<t_k < T, \\ \min_j (t_{j}-t_{j-1})>\delta}} \;
\max_{0\leq j\leq k}w(x,I_j),
\end{eqnarray*}
where $I_j = [t_{j},t_{j+1})$ for $ 0\leq j < k$ and $I_k = [t_{k},T)$. 

Consider a fixed $i\in [n]$.
  Following \cite[Th. 16.8]{billingsley1999convergence} we verify two conditions
  \begin{equation}
      \lim_{a\to\infty} \limsup_{\gamma\downarrow 0} \bP\left( \sup_{t\leq m}\|\inY^{i, \gamma}_t\| \geq a \right) = 0 \quad \forall m\in\bN 
  \end{equation}
\begin{equation}\label{eq:tightness-D-gamma}
    \lim_{\delta\to 0} \limsup_{\gamma\to 0} \bP( w^\prime_m(\inY^{i, \gamma}, \delta)\geq \varepsilon ) = 0 \quad \forall m\in\bN,\, \varepsilon>0
\end{equation}
The first condition is satisfied since $\|\inY^{i, \gamma}\|_\infty \leq R$ for some deterministic $R>0$.
  
We now show that the second condition holds. Fix $\gamma>0,\,\delta \in(0,1/2)$ and $m\in\bN$, define $t_0=0$ and 
$$t_{j} = \inf\{s\geq t_{j-1} + 2\delta: \inY_s^{i, \gamma} \neq Y_{t_{j-1}}^{i, \gamma} \},$$ 
for $j\in \bN$.
Observe that for any $j\in\bN$,
\begin{enumerate}
    \item[(i)] If $t_j> t_{j-1}+2\delta$, then $t_j$ is an update time of $\inY^{i, \gamma}$, that is $\inY^{i, \gamma}_{t_j} = \inX^{i, \gamma}_{t_j}$. 
    Moreover if $t_j> t_{j-1}+2\delta$, then for any $s\in [t_{j-1}+2 \delta, t_j)$ we have $\inY^{i, \gamma}_s = \inY^{i, \gamma}_{t_{j-1}}$.
    \item[(ii)] If $t_j = t_{j-1}+2\delta$, then the last update time up to $t_j$ was on $(t_j-2\delta, t_j]$, that is
    $\inY^{i, \gamma}_{t_j} \in \inX^{i, \gamma}_{(t_j -2\delta,t_j]}$.
    \item[(iii)] For any $s\in [t_{j-1}, t_j)$ we have $\inY^{i, \gamma}_s \in \{\inY^{i, \gamma}_{t_{j-1}}\}\cup \inX^{i, \gamma}_{[t_{j-1},t_j)}$.
\end{enumerate}
From (i) and (ii) it follows that $\inY^{i, \gamma}_{t_j} \in \inX^{i, \gamma}_{(t_j-2\delta, t_j]}$. And from (i) and (iii) it follows that \\
$\inY^{i, \gamma}_s \in \{\inY^{i, \gamma}_{t_{j-1}}\}\cup \inX^{i, \gamma}_{[t_{j-1},t_{j-1}+2 \delta)}$ for any $s\in[t_{j-1}, t_j)$. 
Combining these two observations, we derive that for any $j\in\bN$,
$$
    w( \inY^{i, \gamma}, [t_{j-1},t_j)) \leq w( \inX^{i, \gamma}, (t_{j-1}-2\delta, t_{j-1}+2\delta)).
$$
The partition $(t_0,t_1,\dots, t_k)$ where $k = \max\{j: t_j< m\}$, satisfies $t_{j}-t_{j-1}>\delta, \; \forall 1 \leq j \leq k$. 
Note as well that $Y^{i,\gamma}_{[t_k, m)}\subset Y^{i,\gamma}_{[t_k, t_{k+1})}$ and $t_k+2\delta < m+1$.
Then by definition of $w^\prime_m$,
\begin{equation}\label{eq:modulus-Y-modulus-X-gamma}
    \bP( w^\prime_m(\inY^{i, \gamma}, \delta)\geq \varepsilon ) \leq \bP( w_{m+1}(\inX^{i, \gamma}, 4\delta)\geq \varepsilon) 
\end{equation}
for any $m\in\bN$ and $\delta \in (0,1/2), \,\gamma>0$. The tightness of $\inX^{i, \gamma}$ implies that for any $\varepsilon>0$
$$
\lim_{\delta\to 0} \limsup_{\gamma\to 0} \bP( w_{m+1}(\inX^{i,\gamma}, 4\delta)\geq \varepsilon ) = 0.
$$ 
Hence, using \eqref{eq:modulus-Y-modulus-X-gamma} we recover \eqref{eq:tightness-D-gamma}. \color{black}  \end{itemize}

\end{proof}

\subsection{Integral representation and Brownian motion}

Having established tightness of $(\bs X^\gamma, \bs Y^\gamma, A^\gamma)$, we now study the properties of subsequential limits. From any such limit $(\bs X,\bs Y,A)$, we construct a process $\Theta$ as the time-derivative of $A$, and identify the limiting noise as a Brownian motion with respect to the proper filtration. 

Assume that $\nu^\gamma = \operatorname{Law}(\bs X^\gamma,\bs Y^\gamma, A^\gamma)$ 
converges weakly to a measure $\nu$ along a subsequence as $\gamma\downarrow 0$. 
Let $\pi_{X,Y}:\cC_{d n}\times(\cD_d)^n\times\cC_d \to \cC_{d n}\times(\cD_d)^n$ denote the projection $\pi_{X, Y}(x,y,a) = (x,y)$.
Then $\operatorname{Law}(\bs X^\gamma,\bs Y^\gamma)$ converges weakly to $\mu = (\pi_{X,Y})_\sharp \nu$ along the same subsequence.

Consider the probability space 
\[(\Omega,\cF^0,\bP^0), \quad \Omega=\cC_{dn}\times(\cD_d)^n\times\cC_d, \quad \cF^0=\cB(\Omega), \quad \bP^0=\nu. \] 
Let $(\bs X,\bs Y,A)$ be the canonical process on
$\Omega$, and let $\cF_t^0\eqdef\sigma(\bs X_s,\bs Y_s,A_s:0\leq s\leq
t)$ be its raw natural filtration. Denote by $\cN$ the collection of
all subsets of $\bP^0$-null sets in $\cF^0$. Set
$\cF\eqdef\cF^0\vee\cN$ and
$\cF_t\eqdef\bigcap_{u>t}(\cF_u^0\vee\cN)$, and let $\bP$ be the
completion of $\bP^0$ on $\cF$. Thus,
$(\Omega,\cF,(\cF_t)_{t\geq0},\bP)$ is the usual augmentation of the
canonical filtered probability space. In particular, $(\cF_t)$ is
complete and right-continuous, while $(\bs X,\bs Y,A)$ remains adapted
and has distribution $\nu$.

We define the process $\Theta$ as component-wise limsup of the backward difference
\begin{equation}\label{eq:def-A-prime-left}
    (\Theta_t)^{j} = \limsup_{m\to\infty} m(A_t^j-A_{t-1/m}^j) \,, \quad A = (A^1, \dots, A^d)\,,       
  \end{equation}
  with the convention that $\Theta_0=0$.
  Since every $A^\gamma$ is $R'$-Lipschitz continuous (for a common constant $R'$), and vanishes at
the origin, the same properties hold $\nu$-almost surely for $A$.
Hence, by Rademacher's theorem,
\[
A_t=\int_0^t\Theta_s\,ds,
\qquad t\geq0,
\]
almost surely, where, by definition, $\Theta$ is progressively measurable.

Define now, when $\sigma > 0$, the limiting noise process $\bs B = (B^1,\dots,B^n)$ as
\begin{equation}\label{eq:prop-decomposition}
       B_t^i = \sigma^{-1} \left( \inX_t^i - \inX_0^i -\int_0^t (\beta \Theta_s + \bar \beta  Y^i_s - \inX^i_s) ds \right) ,\quad i\in [n].
\end{equation}

 Our aim is to establish that $\bs B$ is an adapted $dn$-dimensional standard Brownian motion under~$\bP$.
\begin{remark}
When $\sigma=0$, the Brownian motion $\bs B$ does not enter the stochastic differential inclusion. Whenever needed in
Def.~\ref{def:weak-solution}, we enlarge the probability space and take
$\bs B$ to be an arbitrary standard $dn$-dimensional adapted Brownian motion
independent of $(\bs X,\bs Y,A)$. We adopt this convention throughout.
\end{remark}

 Throughout the rest of the section, $(\Omega, \cF, (\cF_t), \bP, \bs X, \bs Y, A, \Theta, \bs B)$ are fixed, as defined above.
\begin{prop}
  \label{prop:brownian}
$\bs B$ is a $dn$-dimensional standard $(\cF_t)$-Brownian motion.
\end{prop}
\begin{proof}
We focus on the case $\sigma>0$. By Eq.~\eqref{eq:decompX}, for every $i\in[n]$,
\[
\sigma^{-1}M^{i,\gamma}
=
\sigma^{-1} \left(X^{i,\gamma}-X_0^{i,\gamma}
-\beta A^\gamma
-\bar\beta\,\iota(Y^{i,\gamma})
+\iota(X^{i,\gamma})
-r^{i,\gamma}\right).
\]
By continuity of $\iota$ (Lemma~\ref{lem:iota-continuous}), the
convergence of $\nu^\gamma$ to $\nu$, the convergence of
$r^{i,\gamma}$ to zero in probability, and the definition \eqref{eq:prop-decomposition} of $\bs B$, we obtain by Slutsky's lemma
$\left(\bs X^\gamma,\bs Y^\gamma,A^\gamma,\sigma^{-1}\bs M^\gamma\right)
\Longrightarrow
(\bs X,\bs Y, A,\bs B)$.
On the other hand, Donsker's theorem shows that
$\sigma^{-1}\bs M^\gamma$ converges weakly to a standard
$dn$-dimensional Brownian motion. Hence $\bs B$ has the law of a
standard Brownian motion. It remains to prove that $\bs B$ is an $(\cF_t)$-martingale.

Consider $0<t<s$.
Define $\cT$ as the set of $u$ such that $\bs Y$ is continuous at $u$ with $\bP$-probability one. Assume first that $t\in \cT$.
Let $\psi_t:(\cD_d)^n \to D([0,t], \bR^d)^n$ be the projection which to every $\bs y$ associates its restriction $\bs y_{[0,t]}$ on the interval $[0,t]$. As $\bs Y$ belongs a.s.
to the continuity set of $\psi_t$, we have by the continuous mapping theorem that
\[
\left(
\sigma^{-1}(\bs M_s^\gamma-\bs M_t^\gamma),
\bs X^\gamma_{[0,t]},
\bs Y^\gamma_{[0,t]},
A^\gamma_{[0,t]}
\right)
\Longrightarrow
\left(
\bs B_s-\bs B_t,
\bs X_{[0,t]},
\bs Y_{[0,t]},
A_{[0,t]}
\right).
\]
Let $\varphi:C([0,t],\bR^{dn})\times (D([0,t],\bR^{d}))^n\times C([0,t],\bR^d)\to\bR$ be bounded and continuous.
Denote $H^\gamma\eqdef \varphi(\bs X^\gamma_{[0,t]},\bs Y^\gamma_{[0,t]},A^\gamma_{[0,t]})$ for each $\gamma>0$.
Upon noting that
$$
\bE\left[
(\bs M_s^\gamma-\bs M_t^\gamma)H^\gamma
\right] = \bE\left[
(\bs M_{\lceil t/\gamma\rceil\gamma}^\gamma-\bs M_t^\gamma)H^\gamma
\right]\,,
$$
we obtain:
\[
\left\|
\bE\left[
(\bs M_s^\gamma-\bs M_t^\gamma)H^\gamma
\right]
\right\|
\leq
\sigma\sqrt{\gamma}\,\|\varphi\|_\infty
\bE\|\bs\xi_{\lceil t/\gamma\rceil}\|
\longrightarrow0 \ \text{ as }\gamma\to 0.
\]
Using also the uniform integrability criterion $\sup_{\gamma\in(0,1)}
\bE\|\bs M_s^\gamma-\bs M_t^\gamma\|^2<\infty$, it follows that
\[
\bE\left[
(\bs B_s-\bs B_t)
\varphi\left(
\bs X_{[0,t]},\bs Y_{[0,t]},A_{[0,t]}
\right)
\right]=0.
\]
This implies that $\bE\left[(\bs B_s-\bs B_t)H\right]=0$ for every
bounded r.v. $H$ measurable w.r.t. the Borel $\sigma$-algebra
of $C([0,t],\bR^{dn})\times D([0,t],\bR^{dn})\times C([0,t],\bR^d)$.
The latter coincides with $\cF^0_t$ by \cite[Th.~12.5.]{billingsley1999convergence}
Thus, 
\[
\bE\left[(\bs B_s-\bs B_t)H\right]=0
\]
for every bounded $(\cF_t^0\vee \cN)$-measurable random variable. 

Set now arbitrary $0\leq t<s$ and let $H$ be bounded and
$\cF_t$-measurable. Since $\mathcal T$ is dense by \cite[pp.138]{billingsley1999convergence}, there exists a
sequence $(t_q)$ in $\mathcal T$ such that
$t_q\downarrow t$ and $t_q<s$. By the definition of the
right-continuous augmentation,
$\cF_t\subset\cF_{t_q}^0\vee\cN$, and therefore
\[
\bE\left[(\bs B_s-\bs B_{t_q})H\right]=0.
\]
Since $\bs B$ has the law of a Brownian motion,
$\bs B_{t_q}\to\bs B_t$ in $L^1$. Letting $q\to\infty$ gives
\[
\bE\left[(\bs B_s-\bs B_t)H\right]=0.
\]
Thus, $\bs B$ is a continuous $(\cF_t)$-martingale. Since it has the
law of a standard Brownian motion, its quadratic covariation satisfies
$[\bs B]_t=tI_{dn}$. Lévy's characterization then shows that $\bs B$
is a standard $dn$-dimensional $(\cF_t)$-Brownian motion \cite[Th.~3.3.16]{shreve1991brownian}.
\end{proof}

\subsection{Properties of the limiting global best}

Recall the definition of  $\Theta$ in Eq.~\eqref{eq:def-A-prime-left}
as the time derivative of $A$, and the definition of $F$ in Eq.~(\ref{eq:F}).
\begin{prop}
  \label{prop:limit-global-best}
  With probability one, $\Theta$ satisfies
  $$
  \Theta_t \in \mathrm{conv}\, F(\bs Y_t)
  $$
  for Lebesgue-almost every $t>0$.
\end{prop}

\begin{proof}

Using Skorokhod's theorem, consider processes $(\tilde{\bs Y}^\gamma, \tilde A^\gamma)$ and
$(\tilde{\bs Y}, \tilde A)$ on some other probability space $(\tilde\Omega, \tilde\cF,\tilde\bP)$, having the same distributions as $(\bs Y^\gamma, A^\gamma)$ and
$(\bs Y,  A)$ respectively, and such that  $(\tilde{\bs Y}^\gamma, \tilde A^\gamma)\to (\tilde{\bs Y}, \tilde A)$
everywhere, in the product topology of $(\cD_d)^n\times \cC_d$, along the same subsequence $\gamma_k$. Throughout this proof
we assume $\gamma\in\{\gamma_k, k\in \bN\}$.

With $\tilde \bP$-probability one, the paths of $\tilde A^\gamma$ are Lipschitz, and thus differentiable Lebesgue-almost everywhere. 
We define $\tilde{\Theta}^\gamma_s(\omega) = \frac{d}{ds}\tilde{A}^\gamma_s(\omega)$ at points $(s, \omega)$ where the derivative exists, 
and $\tilde{\Theta}^\gamma_s(\omega) = 0$ otherwise. This ensures that $(s, \omega) \mapsto \tilde{\Theta}^\gamma_s(\omega)$ is jointly 
measurable in $\cB(\bR_+)\otimes \tilde \cF$, and that almost surely, for all $t \geqslant h > 0$,
\begin{equation*}
    \tilde{A}^\gamma_{t+h} - \tilde{A}^\gamma_{t-h} = \int_{t-h}^{t+h} \tilde{\Theta}^\gamma_s \, ds.
\end{equation*}
Note that $\frac{d}{ds}A^\gamma_s = \Theta^\gamma_s \in F(Y^\gamma_s)$ by definition \eqref{eq:def_Atgamma}. Therefore, the equality in law 
guarantees that $\tilde{\Theta}^\gamma_s \in F(\tilde{Y}^\gamma_s)$ Lebesgue-a.e., with $\tilde \bP$-probability one.
We fix an $\omega\in \tilde \Omega$ such that this property holds for every $\gamma_k$ and omit the dependence on $\omega$ in the sequel.

Let $D_k$ denote the set of points $s$ where $\tilde A^{\gamma_k}$ is differentiable, and let $\operatorname{Cont}(\tilde Y)$ 
be the continuity set of $\tilde Y$. We define 
\[S =  \operatorname{Cont}(\tilde Y)\cap \bigcap_{k\in \bN} D_k.\] Note that because $\tilde Y$ is
càdlàg, $\tilde A^\gamma_k$ are Lipschitz, and the latter intersection is countable, the complement $\bR_+\setminus S$ has zero Lebesgue measure.
Furthermore, since $\tilde Y^{i,\gamma}_s \to \tilde Y^{i}$ in $\cD_d$, for every $i\in[n]$, we have pointwise convergence at any continuity point along $\gamma_k$:
$$
\tilde Y^{i,\gamma}_s \to \tilde Y^i_s, \qquad \forall s \in S .
$$

Let $s\in S$, and let $I_s = \argmin_{i\in[n]} f(\tilde Y^i_s)$. For any $j\notin I_s$, $f(\tilde Y^j_s)> \min_i f(\tilde Y^i_s)$. Since $f$
is continuous and $\tilde Y^{i,\gamma}_s \to \tilde Y^i_s$, there exists $\gamma_0>0$ such that, for any $\gamma < \gamma_0$ and any $j\notin I_s$, $f(\tilde Y^{j,\gamma}_s)> \min_i f(\tilde Y^{i,\gamma}_s)$. Hence, if $\gamma$ is sufficiently small, then $\argmin_i f(\tilde Y^{i,\gamma}_s) \subset I_s$, and therefore 
$\tilde \Theta^\gamma_s \in \{\tilde Y^{i,\gamma}_s\,:\, i \in I_s\}$.
Consequently, the distance
 $d(\tilde \Theta^\gamma_s, F(\tilde{\bs Y}_s)) = \inf_{y\in F(\tilde{\bs Y}_s)} \|\tilde \Theta^\gamma_s - y\|$ 
admits the bound
$$
  d(\tilde \Theta^\gamma_s, F(\tilde{\bs Y}_s)) \leq \max_{i\in I_s} \|\tilde Y^{i,\gamma}_s - \tilde Y^i_s\| .
$$
Thus, $d(\tilde \Theta^\gamma_s, F(\tilde{\bs Y}_s)) \to 0$ along the sequence $\gamma_k$ for any 
$s\in S$.

Let $C_h = \conv\big( \bigcup_{s\in [t-h,t+h]}  F(\tilde{\bs Y}_s)\cup F(\tilde{\bs Y}_s^-) \big)$, where we recall that $\tilde{\bs Y}_s^-$ denotes the left limit of $\tilde{\bs Y}$ at $s$.
Since $C_h$ is convex, $y\mapsto d(y, C_h)$ is a convex continuous function on $\bR^d$. Applying Jensen's inequality, 
and recalling that $[t-h, t+h]\setminus S$ has zero Lebesgue measure, we obtain
\begin{align}\label{eq:dist-A-quotient}
  d\left( \frac{\tilde A^\gamma_{t+h} - \tilde A^\gamma_{t-h}}{2h}, C_h \right) & = 
  d\left(\frac{1}{2h} \int_{[t-h, t+h]} \tilde \Theta^\gamma_s  ds , C_h\right)\\
  &\leq
  \frac{1}{2h} \int_{[t-h, t+h]} d(\tilde \Theta^\gamma_s, C_h) ds
  \\
  &
  = \frac{1}{2h} \int_{[t-h, t+h]\cap S} d(\tilde \Theta^\gamma_s, C_h) ds.
\end{align}
The integrand $d(\tilde \Theta^\gamma_s, C_h)$ is uniformly bounded by 
$2 R^\prime$ and converges to zero pointwise on $S$ along $\gamma_k$, as proven above. 
By dominated convergence theorem, we have
$$
  \lim_{\gamma_k\downarrow 0} \frac{1}{2h} \int_{[t-h, t+h]\cap S} d(\tilde \Theta^\gamma_s, C_h)ds = 0.
$$
It follows that
$
  d\left( \frac{\tilde A_{t+h} - \tilde A_{t-h}}{2h}, C_h \right) = 
  \lim_{\gamma\downarrow 0} d\left( \frac{\tilde A^\gamma_{t+h} - \tilde A^\gamma_{t-h}}{2h}, C_h \right) = 0$. Since $\tilde {\bs Y}$ is càdlàg, $C_h$ contains all its limit points, i.e. it is closed. Hence, we obtain
$$
  \frac{\tilde A_{t+h} - \tilde A_{t-h}}{2h} \in 
  \conv\Bigg( \bigcup_{s\in [t-h,t+h]}  F(\tilde{\bs Y}_s)\cup F(\tilde{\bs Y}_s^-) \Bigg).
$$
Let $t$ be a point of differentiability of $\tilde A$ and $t\in\operatorname{Cont}(\tilde{\bs Y})$. Then, denoting by $\tilde A'_t$ the derivative of $\tilde A$ at this point,
\begin{align*}
  \tilde A_t' &\in \limsup_{h\downarrow 0} \mathrm{conv}\!\!\bigcup_{s\in[t-4h,t+4h]}F(\tilde{\bs Y}_s)\cup F(\tilde{\bs Y}_s^-)\\
  &\subset \mathrm{conv} \left( F(\tilde{\bs Y}_t)\cup F(\tilde{\bs Y}_t^-) \right)\,,
\end{align*}
where the last inclusion is due to Lem.~\ref{lem:small-h}. By continuity of $\tilde{\bs Y}$ at $t$,
$ F(\tilde{\bs Y}_t)\cup F(\tilde{\bs Y}_t^-) =  F(\tilde{\bs Y}_t)$. Then for Lebesgue-a.e. $t>0$, $\tilde A'_t \in \mathrm{conv} F(\tilde{\bs Y}_t)$. 
Recalling that this holds for $\tilde \bP$-a.e. $\omega\in\tilde\Omega$, we conclude that with probability one, 
$\Theta_t = A^\prime_t  \in \mathrm{conv} F(\bs Y_t)$ for Lebesgue-a.e. $t>0$.

\end{proof}


\subsection{Properties of the limiting personal best}

Recall that $(\bs X,\bs Y)$ is the limit in law of the family of processes $(\bs X^\gamma,\bs Y^\gamma)_\gamma$ as $\gamma\downarrow 0$ along a subsequence.
\begin{prop}
  \label{prop:limit-personal-best} With probability one, for all $i\in [n]$,
$$
\forall t\geq 0, \forall s\in [0,t),\ Y^i_t\in\arg\min \{f(u):u\in \{Y^i_s\}\cup X^i([s,t])\}\,.
$$
\end{prop}
\begin{proof}
  We show that the following three statements
hold with probability one, for every $i\in[n]$:
\begin{align}
&Y_0^i=X_0^i, \label{eq:Y0-X0}\\
&\forall t\geq0,\quad
f(Y_t^i)\leq\min_{0\leq s \leq t}f(X_s^i),
\label{eq:Yrecord}\\
&\forall\,0\leq s<t,\quad
Y_t^i\in\{Y_s^i\}\cup X^i((s,t]).
\label{eq:Y-belongs-to-X}
\end{align}
Since $Y_0^{i,\gamma}=X_0^{i,\gamma}$ for every $\gamma>0$ and
evaluation at time zero is continuous on both $\cC_d$ and $\cD_d$,
the convergence $(X^{i,\gamma},Y^{i,\gamma})\Rightarrow(X^i,Y^i)$
implies \eqref{eq:Y0-X0}.

  We set $i\in [n]$, and omit the superscript $i$. We denote $(X,Y)$ for $(X^i,Y^i)$, and $(X^\gamma,Y^\gamma)$ for $(X^{i,\gamma},Y^{i,\gamma})$.
  Let us prove Eq. (\ref{eq:Yrecord}). For any $\delta,T>0$, define
  $$
  \cA_1^{\delta,T}(x) \eqdef \{y\in \cD_d\,:\,\forall 0\leq t< T,\ f(y_t)\leq \min_{s\leq t} f(x_s) +\delta\}.
  $$
  Observe that for any $0\leq t \leq T$,
$$
f(Y^{\gamma}_t) = \min_{0\leq k \leq \lfloor t/\gamma\rfloor} f( X^{\gamma}_{k\gamma}) \leq \min_{s\leq t} f( X^{\gamma}_s) + w_{f\circ  X^{\gamma}, T}(\gamma),
$$
where $ w_{f\circ  X^\gamma, T}(\gamma) = \sup_{s,t\in[0,T],\,|s-t|\leq \gamma} |f( X^\gamma_t) - f( X^\gamma_s)|$. Therefore, $ Y^\gamma \in \cA_1^{\delta, T} ( X^\gamma)$ whenever 
$w_{f\circ  X^{\gamma}, T}(\gamma) \leq \delta$. Since $f$ is continuous and $ X^\gamma$ is tight in $C([0,T],\bR^d)$, $f\circ  X^\gamma$ is tight in $C([0,T],\bR)$, which implies that $\bP(w_{f\circ  X^{\gamma}, T}(\gamma) \geq \delta) \to 0$ as $\gamma\downarrow 0$ for any $\delta>0$.
As a consequence, $$\bP( Y^\gamma \in \cA_1^{\delta,T} ( X^\gamma) ) \geq \bP\big(w_{f\circ  X^{\gamma}, T}(\gamma) \leq \delta\big) \to 1 \text{ as } \gamma \downarrow 0, \quad \forall \delta>0\, .
$$
By Lem.~\ref{lem:A1-closed}, $\graph(\cA_1^{\delta,T})$ is closed. By the Portmanteau theorem, $\bP( Y \in \cA_1^{\delta,T} ( X) )= 1$ for every $\delta>0$.
Thus, $Y \in \bigcap_{\delta\in 1/\bN^*, T\in \bN^*}\cA_1^{\delta,T} ( X)$ a.s. We have that (\ref{eq:Yrecord}) holds w.p.1.

Now consider the statement~(\ref{eq:Y-belongs-to-X}) (we still omit the superscript $i$).
  By construction of $Y^{\gamma}$, the condition stated in the proposition is satisfied for $(X^{\gamma},Y^{\gamma})$:
  \[\forall t\geq 0, \, \forall s\in [0,t),\;\; Y^{\gamma}_t \in \{Y^{\gamma}_s\}\cup X^{\gamma}((s,t]).
  \]
  Indeed, if $Y_t \neq Y_s$, the position stored at time $t$ must have been selected at a grid time belonging to $(s,t]$.
  By Lem.~\ref{lem:A2-closed} and the fact that $(X^{\gamma},Y^{\gamma})\Rightarrow (X,Y)$,
  one can pass to the limit as $\gamma\downarrow 0$, by the Portmanteau theorem. This completes the proof.
  
  It remains to show why these properties imply the statement of
the proposition. Taking $s=0$ in
\eqref{eq:Y-belongs-to-X} and using $Y_0^i=X_0^i$, we obtain $Y_t^i\in X^i([0,t])$.
Together with \eqref{eq:Yrecord}, this yields
\[
f(Y_t^i)=\min_{0\leq r\leq t}f(X_r^i),
\qquad t\geq0.
\]
Therefore, $t\mapsto f(Y_t^i)$ is nonincreasing. Hence, for
$0\leq s<t$,
\[
f(Y_t^i)\leq f(Y_s^i)
\quad\text{and}\quad
f(Y_t^i)\leq\min_{r\in[s,t]}f(X_r^i).
\]
Combining these inequalities with
\eqref{eq:Y-belongs-to-X} allows to conclude.
\end{proof}

\subsection{End of the proof of Th.~\ref{th:sdi}}

The proof is completed by putting together Eq.~(\ref{eq:prop-decomposition}), Prop.~\ref{prop:tightness-gamma}, Prop.~\ref{prop:brownian}, Prop.~\ref{prop:limit-global-best} and Prop.~\ref{prop:limit-personal-best}.

\section{Recovery of the personal- and global-best selectors (Th.~\ref{th:continuous-time-limit})}
\label{s:girsanov}

We now use the additional nondegeneracy assumptions to recover
the exact selectors from the relaxed characterization of Th.~\ref{th:sdi}, thereby proving
Th.~\ref{th:continuous-time-limit}. This crucially relies on the presence of Brownian noise (hence $\sigma > 0$). We describe below our standing assumptions for this section.

\begin{assumption}[Standing Assumption]\label{hyp:standing}
Assumptions~\ref{hyp:stab} and~\ref{hyp:random} are in force.
Moreover, consider a filtered probability space
$(\Omega,\cF,(\cF_t)_{t\geq0},\bP)$ 
supporting an adapted process $(\bs X,\bs Y)$ with paths in
$\cC_{dn}\times(\cD_d)^n$, a progressively measurable process
$\Theta$ with values in $\bR^d$, and a standard $dn$-dimensional
$(\cF_t)$-Brownian motion $\bs B$. We assume that $(\Omega,\cF,(\cF_t)_{t\geq0},\bP,\bs X,\bs Y,\Theta,\bs B)$
satisfy statements i)--iv) of Th.~\ref{th:sdi}, and that $(\cF_t)_{t \geq 0}$ satisfies the usual conditions.
\end{assumption}
In particular, the additional assumptions that we make with respect to the previous section are that $d\geq2$, $\sigma>0$,  $f$ has isolated
critical points, and $\mu_0$ is absolutely continuous with
respect to Lebesgue measure on $\bR^{dn}$.

\subsection{Avoidance of critical points}

\begin{prop}
\label{prop:avoid-critical-points}
Under Assumption~\ref{hyp:standing}, with probability one,
$X_t^i\notin\crit f$,
$\forall i\in[n],\ \forall t\geq0$.
\end{prop}
\begin{proof}
  We remove the drift by a localized change of measure, apply the fact that Brownian motion in dimension at least two avoids prescribed points, and then transfer this property back to the original measure.
  Let $i\in [n]$ be fixed. We omit the index $i$ in the sequel, and just write $(X_t,Y_t,B_t)$ instead of $(X_t^i,Y_t^i,B^i_t)$.
  Set $b_t
\eqdef
\beta\Theta_t+\bar\beta Y_t-X_t$.
The process $b$ is progressively measurable, and bounded on every compact time interval.
For $m\in\bN^*$, define the stopping time
$\rho_m
\eqdef
m\wedge
\inf\left\{
t\geq0:
\int_0^t\|b_s\|^2\,ds\geq m
\right\}$. Set
\[
H_s^{(m)}
\eqdef
-\frac1\sigma b_s\un_{\{s\leq\rho_m\}},
\qquad 0\leq s\leq m.
\]
Note that $
\int_0^m\|H_s^{(m)}\|^2\,ds
=
\frac1{\sigma^2}
\int_0^{\rho_m}\| b_s\|^2\,ds
\leq\frac m{\sigma^2}$.
Therefore,
\[
\bE\left[
\exp\left(
\frac12\int_0^m\|H_s^{(m)}\|^2\,ds
\right)
\right]
\leq
\exp\left(\frac{m}{2\sigma^2}\right)<\infty.
\]
Novikov's criterion, applied on the time interval
$[0,m]$ to the integrand $H^{(m)}$, shows that
\[
L_t^{(m)}
=
\exp\left(
\int_0^t(H_s^{(m)})^\top\,d B_s
-\frac12\int_0^t\|H_s^{(m)}\|^2\,ds
\right),
\qquad 0\leq t\leq m,
\]
is a martingale. In particular,
$\bE[L_m^{(m)}]=L_0^{(m)}=1$.
Define a probability measure $\bQ_m$ on
$(\Omega,\cF_m)$ by $\bQ_m(A) \eqdef \bE(\bs 1_A L_m^{(m)})$.
Since $L_m^{(m)}>0$ almost surely, $\bQ_m$ and $\bP$ are equivalent
on $\cF_m$.
By Girsanov's theorem, the process
$\widetilde{B}_t^{(m)}
\eqdef
B_t-\int_0^t H_s^{(m)}ds$, defined for $t\in [0,m]$,
is a standard $d$-dimensional $(\cF_t)$-Brownian motion under
$\bQ_m$. In particular, it is independent of $\cF_0$ under
$\bQ_m$ \cite[Th. 5.1]{shreve1991brownian}.
Since $d\geq2$, every point of $\bR^d$ is polar for this
process \cite[Cor~2.26]{morters2010brownian}.  Using that $\crit f$ is countable,
\begin{equation}
\forall x\notin\crit f,\ \bQ_m\left(
\exists t\in[0,m]:
x+\sigma\widetilde B_t^{(m)}\in\crit f
\right)
=0.\label{eq:nocrit-Btilde}
\end{equation}
From statement ii) of Th.~\ref{th:sdi}, for every $0\leq t\leq \rho_m$, $X_t=X_0+\sigma\widetilde B_{t}^{(m)}$.
Define $$
p_m \eqdef \bQ_m\left(\exists t\in[0,\rho_m]:X_t\in\crit f\right)\,,
$$
and let us prove that $p_m=0$. By disintegration, and using the fact that $\widetilde B$ is independent of $\cF_0$,
\begin{align*}
p_m 
      &=  \bQ_m\left(\exists t\in[0,m]:X_0+\sigma\widetilde B_{t\wedge \rho_m}^{(m)}\in\crit f\right) \\
    &\leq  \bQ_m\left(\exists t\in[0,m]:X_0+\sigma\widetilde B_{t}^{(m)}\in\crit f\right)\\
      &=  \int \bQ_m\left(\exists t\in[0,m]:x+\sigma\widetilde B_{t}^{(m)}\in\crit f\right)d\bQ_m(X_0\in dx) \\
  &\leq  \bQ_m(X_0\in \crit f)\,,
\end{align*}
where the last inequality uses Eq.~(\ref{eq:nocrit-Btilde}). Moreover, $\bQ_m$ and $\bP$ coincide on $\cF_0$,
because of the martingale property of $L^{(m)}$. Thus, $p_m\leq \bP(X_0\in\crit f)$, which is equal to zero, because
$\crit f$ is at most countable, and the law of $X_0$ is absolutely continuous with
respect to Lebesgue measure on $\bR^d$.
As $\bQ_m$ and $\bP$ are equivalent on $\cF_m$, $\bP\left(\left\{
\exists t\in [0,\rho_m]:X_t\in\crit f
\right\}\right)=0$. Since $\rho_m\uparrow\infty$ almost surely as $m\to\infty$, the proof is concluded by
\[
\bP\left(\left\{
\exists t\geq0:X_t\in\crit f
\right\}\right)
= \lim_{m\to\infty} \bP\left(\left\{
\exists t\in [0,\rho_m]:X_t\in\crit f
\right\}\right)=0\,.
\]
\end{proof}

\subsection{Recovery of the personal best selector}

We start by proving that, at a given stopping time, the function value of $X$ almost surely instantly improves, due to Brownian noise. This is a consequence of the law of the iterated logarithm.

\begin{lemma}[Immediate improvement at stopping times]
\label{lem:infinite-crossing}
Under Assumption~\ref{hyp:standing}, let $\tau$ be an
$(\cF_t)$-stopping time. Then, with probability one, either
$\tau=\infty$, or, for every $i\in[n]$
\[
\forall \varepsilon > 0, \qquad
f(X_\tau^i)
>
\min_{s\in[\tau,\tau+\varepsilon]}f(X_s^i).
\]
\end{lemma}

We deduce Lemma~\ref{lem:infinite-crossing} from the following
more general result, which we will also apply to an auxiliary
Brownian process in Section~\ref{ss:global_best_selector}.

\begin{lemma}
\label{lem:infinite-crossing-generic}
Let $\sigma>0$, and $f$ twice continuously differentiable on $\bR^d$. 
Let $(\Omega,\cF,(\cF_t)_{t\geq0},\bP)$ be a filtered probability
space satisfying the usual conditions, and let $B$ be a standard
$d$-dimensional $(\cF_t)$-Brownian motion. Let $X$ be an
$(\cF_t)$-adapted continuous process and let $Z$ be an
$(\cF_t)$-progressively measurable process such that
\[
\operatorname*{ess\,sup}_{0\leq s\leq T}\|Z_s\|<\infty
\qquad\text{a.s., for every }T>0,
\]
and, almost surely, for every $t\geq0$,
\[
X_t
=
X_0+\int_0^t(Z_s-X_s)\,ds+\sigma(B_t-B_0).
\]
Assume moreover that $\bP\bigl(X_t\notin\crit f\text{ for every }t\geq0\bigr)=1$.
Let $\tau$ be an $(\cF_t)$-stopping time. Then, with probability one,
either $\tau=\infty$, or
\[
\forall\varepsilon>0,\qquad
f(X_\tau)
>
\min_{s\in[\tau,\tau+\varepsilon]}f(X_s).
\]
\end{lemma}

\begin{proof}
For $m\in\bN^*$, let $\rho_m=\tau\wedge m$ and set
\[
U_m\eqdef
\frac{\nabla f(X_{\rho_m})}
     {\|\nabla f(X_{\rho_m})\|},
\]
with an arbitrary fixed unit vector as its value when the
denominator vanishes. By assumption, this exceptional event
has probability zero. $U_m$ is
$\cF_{\rho_m}$-measurable.

By the strong Markov property, $ \widehat B^{\,m} 
\eqdef (B_{\rho_m+h}-B_{\rho_m})_{h\geq0}$
is a standard $d$-dimensional Brownian motion independent of
$\cF_{\rho_m}$. Therefore, conditionally on
$\cF_{\rho_m}$, the process
$U_m^\top\widehat B^{\,m}$ is a standard one-dimensional
Brownian motion.
The scalar law of the iterated logarithm gives, setting $a(h)\eqdef \sqrt{2h\log\log(1/h)}$,
\[
\liminf_{h\downarrow0}
\frac{U_m^\top\widehat B_h^{\,m}}{a(h)}
=-1, \qquad \limsup_{h\downarrow0}
\frac{U_m^\top\widehat B_h^{\,m}}{a(h)}
=1,
\]  
with probability one; see \cite[Cor.~5.3]{morters2010brownian}.
Applying the same law to each coordinate also gives
\[
\|\widehat B_h^{\,m}\|=O(a(h))
\qquad\text{as }h\downarrow0,
\quad\text{almost surely}.
\]
Intersecting these probability-one events over
$m\in\bN^*$, and using $\rho_m=\tau$ on $\{\tau\leq m\}$,
we obtain, almost surely on $\{\tau<\infty\}$,
\[
\liminf_{h\downarrow0}
\frac{U^\top(B_{\tau+h}-B_\tau)}{a(h)}
=-1,
\qquad
\|B_{\tau+h}-B_\tau\|=O(a(h)),
\]
where $g\eqdef\nabla f(X_\tau)$ and $U\eqdef g/\|g\|$. 
Here, $g\neq0$ by the critical-point avoidance assumption.

We now work pathwise on this probability-one event and on
$\{\tau<\infty\}$.
Write
\[
X_{\tau+h}-X_\tau
=\sigma(B_{\tau+h}-B_\tau)+R_h,
\qquad
R_h\eqdef\int_\tau^{\tau+h}(Z_s-X_s)\,ds.
\]
The local essential boundedness of $Z$ and continuity of $X$
imply that $\|R_h\|=O(h)$. 
Since $h=o(a(h))$, it follows that $$\|X_{\tau+h}-X_\tau\|=O(a(h)).$$ 
As $f$ is $C^2$, Taylor's formula yields
\begin{align*}
f(X_{\tau+h})-f(X_\tau)
&=g^\top(X_{\tau+h}-X_\tau)
  +O(\|X_{\tau+h}-X_\tau\|^2)\\
&=\sigma g^\top(B_{\tau+h}-B_\tau)
  +O(h)+O(a(h)^2).
\end{align*}
All the $O(\cdot)$ estimates here are pathwise.
Dividing by $a(h)$, using
$h/a(h)\to0$ and $a(h)\to0$, we obtain
\[
\liminf_{h\downarrow0}
\frac{f(X_{\tau+h})-f(X_\tau)}{a(h)}
=-\sigma\|g\|<0.
\]
Hence for every $\varepsilon >0$, there exists $h \in [0, \varepsilon]$ such that $f(X_{\tau+h}) < f(X_\tau)$, which proves the claim.
\end{proof}

\begin{proof}[Proof of Lemma~\ref{lem:infinite-crossing}]
  Fix $i\in[n]$ and set $
Z^i_t\eqdef\beta\Theta_t+\bar\beta Y_t^i$.
Statement iv) of Th.~\ref{th:sdi}, together with coercivity of $f$ and boundedness of the support of $\mu_0$, implies that $Y^i$ takes values in a deterministic compact set. Statement iii) then implies that $\Theta$, and hence $Z^i$, is almost surely essentially bounded. Prop.~\ref{prop:avoid-critical-points} shows that
$X_t^i\notin\crit f$ for every $t\geq0$, almost surely.
Thus, Lemma~\ref{lem:infinite-crossing-generic} applies to
$X^i$. Applying it for every $i\in[n]$ and intersecting the
results allows to conclude.
\end{proof}

We are now ready to prove the recovery of the personal best selector, defined in Eq.~\eqref{eq:pbest}.

\begin{prop}
\label{prop:pbest}
Under Assumption~\ref{hyp:standing}, with probability one,
\[
Y_t^i=\pbest(X^i)_t,
\qquad
\forall i\in[n],\ \forall t\geq0.
\]
\end{prop}

\begin{proof}
Fix $i\in[n]$ and omit the superscript $i$.
For every $\delta\in1/\bN^*$, set $t_0^\delta=0$ and define
recursively, for $j\geq1$,
\[
t_j^\delta
\eqdef
\inf\left\{
t\geq t_{j-1}^\delta+\delta:
f(X_t)=\min_{0\leq s\leq t}f(X_s)
\right\},
\]
with the convention $\inf\varnothing=\infty$. The process
$
f(X_t)-\min_{0\leq s\leq t}f(X_s)
$
is adapted and has continuous paths. Consequently, each
$t_j^\delta$ is a stopping time.
Since the family
$\{t_j^\delta:\delta\in1/\bN^*,\,j\geq0\}$ is countable, we may work
on a probability-one event on which statement iv) of
Th.~\ref{th:sdi} holds and, for every
$\delta\in1/\bN^*$ and every $j\geq0$, either
$t_j^\delta=\infty$, or
\[
\forall\varepsilon>0,\qquad
f(X_{t_j^\delta})
>
\min_{s\in[t_j^\delta,t_j^\delta+\varepsilon]}f(X_s),
\]
by Lemma~\ref{lem:infinite-crossing}. Let $t_j^\delta<\infty$.
Since $t_j^\delta$ is a record time and $Y_0=X_0$, statement iv) gives
$f(Y_{t_j^\delta})=f(X_{t_j^\delta})$. For every
$\varepsilon>0$, Lemma~\ref{lem:infinite-crossing} provides
$s_\varepsilon\in(t_j^\delta,t_j^\delta+\varepsilon]$ such that
$f(X_{s_\varepsilon})<f(Y_{t_j^\delta})$. Applying statement iv)
between $t_j^\delta$ and $s_\varepsilon$, we obtain
$Y_{s_\varepsilon}=X_{u_\varepsilon}$ for some
$u_\varepsilon\in[t_j^\delta,s_\varepsilon]$. Letting
$\varepsilon\downarrow0$, the right continuity of $Y$ and the
continuity of $X$ yield
$Y_{t_j^\delta}=X_{t_j^\delta}$.

We now show that $Y_t=X_t$ at every record time $t$. This is already
proved if $t=t_j^\delta$ for some $\delta\in1/\bN^*$ and $j\geq0$.
Otherwise, for each $\delta\in1/\bN^*$, let
$\tau^\delta$ be the largest selected record time $t_j^\delta<t$.
Such a time exists since $t_0^\delta=0$, and the definition of the
sequence $(t_j^\delta)_j$ implies
$t-\tau^\delta<\delta$. By statement iv),
$Y_t\in\{Y_{\tau^\delta}\}\cup X([\tau^\delta,t])$. Since
$Y_{\tau^\delta}=X_{\tau^\delta}$, there exists
$u^\delta\in[\tau^\delta,t]$ such that $Y_t=X_{u^\delta}$.
As $\delta\downarrow0$, $u^\delta\to t$, and continuity of $X$ gives
$Y_t=X_t$.

Set now $t\geq0$ and
$\tau_t=\sup\arg\min_{0\leq r\leq t}f(X_r)$. The time $\tau_t$ is a
record time, hence $Y_{\tau_t}=X_{\tau_t}$. If $\tau_t=t$, then
$Y_t=X_{\tau_t}$. If $\tau_t<t$, statement iv), applied between
$\tau_t$ and $t$, shows that $Y_t$ minimizes $f$ over
$\{X_{\tau_t}\}\cup X([\tau_t,t])$. By definition of $\tau_t$, no
time in $(\tau_t,t]$ attains the value $f(X_{\tau_t})$; therefore
$Y_t=X_{\tau_t}=\pbest(X)_t$.
\end{proof}

\subsection{Recovery of the global best selector}\label{ss:global_best_selector}

Having identified the personal-best processes, it remains to prove that the set of global best particles is a singleton, for almost every $t$ and almost surely. This is achieved by an argument relying on Brownian noise and Girsanov's theorem.

\begin{prop}
  \label{prop:different-f}
  Let $i,j \in [n]$ such that $i\neq j$. Under Assumption~\ref{hyp:standing}, the following holds with probability one:
  $f(Y^i_t)\neq f(Y^j_t)$ for Lebesgue-almost every $t>0$.
\end{prop}
This proposition immediately leads to the following corollary. The proof of the corollary is omitted.
\begin{coro}  \label{cor:gbest}
  Under Assumption~\ref{hyp:standing}, with probability one, $F(\bs Y_t)$ is a singleton for Lebesgue-almost every $t\geq 0$.
  In particular, $F(\bs Y_t) = \{\Theta_t\}$ $\bP\otimes dt$-a.e.
\end{coro}

\begin{proof}[Proof of Prop.~\ref{prop:different-f}]
Set $i,j\in[n]$ with $i\neq j$, and set
$m_t^k\eqdef\min_{0\leq s\leq t}f(X_s^k)$, for all $k\in[n]$.
By Prop.~\ref{prop:pbest}, $m_t^k=f(Y_t^k)$ for every $t\geq0$,
almost surely.
We first introduce the following localized change of measure, similarly to the one of Prop.~\ref{prop:avoid-critical-points}. Write
\[
\bs b_t
\eqdef
\bigl(\beta\Theta_t+\bar\beta Y_t^k-X_t^k\bigr)_{k\in[n]}
\]
and, for $m\in\bN^*$, define
$\rho_m
\eqdef
m\wedge
\inf\left\{
t\geq0:
\int_0^t\|\bs b_s\|^2\,ds\geq m
\right\}$ and $\bs H_t^{(m)}
\eqdef
-\frac1\sigma\bs b_t\un_{\{t\leq\rho_m\}}$.
As in the proof of Prop.~\ref{prop:avoid-critical-points}, the
following stochastic exponential, defined for all $t\geq 0$,
\[
L_t^{(m)}
\eqdef
\exp\left(
\int_0^t (\bs H_s^{(m)})^\top d\bs B_s
-\frac12\int_0^t\|\bs H_s^{(m)}\|^2\,ds
\right),
\]
is a martingale. Note that $L_t^{(m)} = L_m^{(m)}$ for $t\geq m$.
Define $\bQ_m$ on $(\Omega,\cF)$ by
$\bQ_m(A)
\eqdef
\bE\bigl[\un_A L_m^{(m)}\bigr]$.
Applying
Girsanov's theorem on every deterministic interval $[0,T]$ shows that
\[
\widetilde{\bs B}_t^{(m)}
\eqdef
\bs B_t-\int_0^t\bs H_s^{(m)}\,ds
\]
is a standard $dn$-dimensional $(\cF_t)$-Brownian motion under $\bQ_m$, independent of $\cF_0$.
For each $k\in [n]$, and each $t\in [0,m]$, 
\[
X_{t\wedge \rho_m}^k
=
X_0^k+\sigma\widetilde B_{t\wedge\rho_m}^{k,(m)}.
\]
Set $t>0$ and $m>t$. Let $c\in\bR$, and choose
$x\in\bR^d\setminus\crit f$. Consider the event
\[
A_c^{k,x}
\eqdef
\left\{
\min_{0\leq s\leq t}
f\bigl(x+\sigma\widetilde B_s^{k,(m)}\bigr)=c
\right\}.
\]
This event is empty if $c>f(x)$, so assume that $c\leq f(x)$. Define
\[
T_c
\eqdef
\inf\left\{
s\in[0,t]:
f\bigl(x+\sigma\widetilde B_s^{k,(m)}\bigr)\leq c
\right\},
\]
with the convention $\inf\varnothing=\infty$. Note that $A_c^{k,x}\subset \{T_c\in[0,t]\}$. Also, by continuity,
$f\bigl(x+\sigma\widetilde B_{T_c}^{k,(m)}\bigr)=c$.
Under $\bQ_m$, the process
\[
V_s^x\eqdef x+\sigma\widetilde B_s^{k,(m)},
\]
is a Brownian motion with covariance matrix $\sigma^2 I_d$. Since
$d\geq2$ and $\crit f$ is countable, $V^x$ almost surely never
visits $\crit f$. Moreover, it satisfies
\[
V_s^x
=
V_0^x+\int_0^s(V_r^x-V_r^x)\,dr
+\sigma\bigl(\widetilde B_s^{k,(m)}
             -\widetilde B_0^{k,(m)}\bigr).
\]
Thus, Lemma~\ref{lem:infinite-crossing-generic}, applied under
$\bQ_m$ to $V^x$ at the stopping time $T_c$, implies that, almost
surely,
\[
\forall\varepsilon>0,\ 
f(V_{T_c}^x)
>
\min_{s\in[T_c,T_c+\varepsilon]}f(V_s^x).
\]
On the event $\{T_c <t\}$, choosing $\varepsilon\in(0,t-T_c)$ yields
\[
\min_{0\leq s\leq t}f(V_s^x)
<
f(V_{T_c}^x)
=c.
\]
This is incompatible with $A_c^{k,x}$. We have shown that $A_c^{k,x} \cap \{T_c<t\}$ is a $\bQ_m$-negligible set.
Consequently,
\[
\bQ_m(A_c^{k,x})\leq \bQ_m(\{T_c=t\}) \leq \bQ_m\left( f(x+\sigma\widetilde B_t^{k,(m)}) = c\right).
\]
As $t>0$, the law of $x+\sigma\widetilde B_t^{k,(m)}$ is absolutely continuous w.r.t. the Lebesgue measure on $\bR^d$.
Moreover, the set $f^{-1}(\{c\})$ has zero Lebesgue measure, by the same argument as in the proof of Proposition~\ref{prop:unique-best-discrete}. 
We conclude that $\bQ_m(f(x+\sigma\widetilde B_t^{k,(m)})= c)=0$. Therefore, $\bQ_m(A_c^{k,x})=0$, for every $c\in \bR$, and every $x\notin \crit f$.
Define
\begin{align*}
  p_m &\eqdef \bQ_m\left(\{f(Y_t^i) = f(Y_t^j)\}\cap\{\rho_m\geq t\}\right) \\
      &= \bQ_m\left(
      \left\{\min_{0\leq s \leq t}  f(X_0^i+\sigma\widetilde B_s^{i,(m)})
      = \min_{0\leq s \leq t}  f(X_0^j+\sigma\widetilde B_s^{j,(m)})\right\}
      \cap\{\rho_m\geq t\}\right) \\
    &\leq \bQ_m\left(\min_{0\leq s \leq t}  f(X_0^i+\sigma\widetilde B_s^{i,(m)}) = \min_{0\leq s \leq t}  f(X_0^j+\sigma\widetilde B_s^{j,(m)})\right) \\
    &= \int \bQ_m\left(\min_{0\leq s \leq t}  f(x+\sigma\widetilde B_s^{i,(m)}) = \min_{0\leq s \leq t}  f(y+\sigma\widetilde B_s^{j,(m)})\right)\bP((X_0^i,X_0^j)\in dx\times dy)\,, 
\end{align*}
where we used disintegration, along with the fact that $\widetilde{\bs B}^{(m)}$ is independent of $\cF_0$, and that $\bQ_m$ coincides with $\bP$ on $\cF_0$.
As $\bP((X_0^i,X_0^j)\in \,\cdot\,)$ is absolutely continuous w.r.t. the Lebesgue measure on $\bR^{2d}$, and since $\crit f$ is countable,
the integral can be restricted to the set $(\bR^d\setminus\crit f) \times(\bR^d\setminus\crit f)$. But on that set, the integrand is zero.
Indeed, for a fixed pair $(x,y)$ on non critical points, $\widetilde B^{i,(m)}$ is independent of $\widetilde B^{j,(m)}$. Thus,
the integrand writes itself as the integral of $c\mapsto \bQ_m(A_c^{i,x})$ w.r.t. the law of $ \min_{0\leq s \leq t}  f(y+\sigma\widetilde B_s^{j,(m)})$.
Recalling $\bQ_m(A_c^{i,x})=0$, we conclude that $p_m=0$.
By the equivalence of $\bQ_m$ and $\bP$, it holds that, for every $m$, 
$\bP(f(Y_t^i) = f(Y_t^j), \rho_m\geq t)=0$.
Using that $\rho_m\to\infty$ as $m\to\infty$, it follows that $\bP(f(Y_t^i) = f(Y_t^j))=0$.
Finally, Tonelli's theorem gives
\[
\bE\left[
\int_0^\infty
\un_{\{f(Y_t^i)=f(Y_t^j)\}}\,dt
\right]
=
\int_0^\infty
\bP\bigl(f(Y_t^i)=f(Y_t^j)\bigr)\,dt
=0.
\]
Hence, with probability one,
$f(Y_t^i)\neq f(Y_t^j)$ for Lebesgue-almost every $t>0$.
\end{proof}

\subsection{End of proof of Th.~\ref{th:continuous-time-limit}}

Combine Th.~\ref{th:sdi}, Prop.~\ref{prop:pbest} and Cor.~\ref{cor:gbest}.

\section{Technical lemmas and omitted proofs}
\label{s:technical}

\begin{proof}[Proof of Proposition~\ref{prop:unique-best-discrete}]
    Let $k\geq 0$. Since $\sigma>0$ and the initial distribution $\mu_0$ is absolutely continuous w.r.t. the Lebesgue measure,
  a straightforward induction shows
that the joint distribution of
$(x_l^i:i\in[n],\ 0\leq l\leq k)$
has a density with respect to Lebesgue measure on
$\bR^{dn(k+1)}$. In particular, for arbitrary $i\neq j$ and
$l,m\leq k$, the random vector $(x_l^i,x_m^j)$ has a density with
respect to Lebesgue measure on~$\bR^{2d}$.

  Set $i\neq j$.  Note that the event $f(y^i_k) = f(y^j_k)$ is included in the union of events $\{f(x_l^i) = f(x_m^j)\}$
  over all indices $l,m\leq k$. It is thus sufficient to prove that $\bP(f(x_l^i) = f(x_m^j)) = 0$ for any fixed $(l,m)$.
  The proof is concluded as long as one shows that the set $g^{-1}(\{0\})$ is Lebesgue negligible, where  $g(x,y)\eqdef f(x)-f(y)$.
  Since the critical points of $f$ are isolated, $\crit g$ is at most countable. It is thus sufficient to prove that
  $g^{-1}(\{0\}) \setminus \crit g$ is Lebesgue negligible.
By the implicit function theorem, this forms a $(2d-1)$-dimensional manifold in $\bR^{2d}$, and therefore has zero Lebesgue measure (see
\cite[Th.~9.28]{rudin1976principles}). 
\end{proof}

\begin{proof}[Proof of Lemma~\ref{lem:pbest-well-defined}]
  For every $x\in\cC_d$, $\tau_t(x)$ is well defined for each $t$, and so is $x_{\tau_t(x)}$. We justify that  $x_{\tau_t(x)}$ is càdlàg,
  which makes $\pbest$ well defined. To this end, let $t$ be fixed, and denote $\tau \eqdef \tau_t(x)$ and $\delta\eqdef f(x_t)-f(x_\tau)$. 
  For any $\varepsilon>0$, we have $\tau\leq \tau_{t+\varepsilon}(x)\leq t+\varepsilon$. If $\delta=0$, then $\tau=t$, and thus 
  $\lim_{\varepsilon\downarrow 0} \tau_{t+\varepsilon}(x) = t = \tau$. If $\delta>0$, there exists 
  $\varepsilon>0$, such that $ \max_{s \in [t, t+\varepsilon]} |f(x_s)-f(x_t)| < \delta$. Then $\tau_s(x) = \tau_t(x)$ for any $s\in[t,t+\varepsilon]$, 
  so again, $\lim_{\varepsilon\downarrow 0} \tau_{t+\varepsilon}(x) = \tau$.
  Hence, we have shown that $t\mapsto\tau_t(x)$ is right-continuous, and so is $\pbest(x)$ by continuity of $x$.
  Let $s_k \uparrow t$. Then $\tau_{s_k}(x)$ is non-decreasing and bounded above, so $\tau_{s_k}$ admits a limit. Thus, $\tau_{t-}(x)$
  exists for every $t$. Consequently, $t \mapsto \pbest(x)_t$ admits a left limit at every $t$. 
  
  We have shown that $\pbest$ is well defined.
  It is clearly non anticipative. It remains to establish measurability. For this, it is enough to establish the measurability of the coordinate
  $x\mapsto \pbest(x)_t$ for any fixed $t>0$ (\cite[Th. 16.6]{billingsley1999convergence}). As measurability is a property closed under composition,
  it is sufficient to show $x\mapsto \tau_t(x)$ is measurable, for fixed $t$.

  Consider a fixed $c \in (0,t]$. The inequality $\tau_t(x) < c$ holds iff the
  minimum of $\min_{s\in[0,t]} f(x_s)$ is achieved strictly before
  $c$, and is not achieved on the interval $[c, t]$.  That is, the
  inequality is equivalent to
  \begin{equation*}
      \min_{s \in [c, t]} f(x_s) > \min_{s \in [0, t]} f(x_s) \,.
  \end{equation*}
Denoting $g_1(x) \eqdef \min_{s \in [c, t]} f(x_s)$ and $g_2(x) \eqdef \min_{s \in [0, t]} f(x_s)$, 
we obtain that
\begin{equation*}
      \{ x \in \cC_d : \tau_t(x) < c \} = \{ x \in \cC_d : g_1(x) - g_2(x) > 0 \} \,.
  \end{equation*}
By continuity of $f$, the functions $g_1$ and $g_2$ are continuous. 
  Thus, $\{ x \in \cC_d : \tau_t(x) < c \}$ is open. This proves that $x\mapsto \tau_t(x)$ is measurable. This completes the proof.
\end{proof}

\begin{lemma}\label{lem:iota-continuous}
  The map $\iota$ defined in Eq.~(\ref{eq:def-iota}) is continuous on $\cD_d\to\cC_d$ (and thus on $\cC_d\to\cC_d$).
\end{lemma}
\begin{proof}
Let $y^n\to y$ in $\cD_d$ and fix $T>0$. By the definition of the
Skorokhod topology, there exist increasing bijections $\lambda_n$
such that, on a compact interval containing $[0,T]$,
\[
\|\lambda_n-I\|_\infty\longrightarrow0,
\qquad
\|y^n\circ\lambda_n-y\|_\infty\longrightarrow0.
\]
Consequently, $\lambda_n^{-1}\to I$ uniformly and
$\sup_{t\leq T}
\|y^n(t)-y(\lambda_n^{-1}(t))\|
\longrightarrow0.$
Therefore,
\begin{align*}
\sup_{t\leq T}
\|\iota(y^n)_t-\iota(y)_t\|
&\leq
\int_0^T\|y^n(s)-y(s)\|\,ds\\
&\leq
T\sup_{s\leq T}
\|y^n(s)-y(\lambda_n^{-1}(s))\|
+
\int_0^T
\|y(\lambda_n^{-1}(s))-y(s)\|\,ds.
\end{align*}
The first term converges to zero. The second one also converges to
zero by dominated convergence, since $y$ is locally bounded and is
continuous Lebesgue-almost everywhere. Hence
$\iota(y^n)\to\iota(y)$ uniformly on every compact time interval.
\end{proof}

\begin{lemma}
\label{lem:lim-F-leftright}
Let $f$ be continuous and let $\bs y\in (\cD_d)^n$. For every
$t>0$,
\begin{align*}
\limsup_{s\downarrow t}
\bigl(F(\bs y_s)\cup F(\bs y_{s-})\bigr)
&\subset F(\bs y_t),\\
\limsup_{s\uparrow t}
\bigl(F(\bs y_s)\cup F(\bs y_{s-})\bigr)
&\subset F(\bs y_{t-}).
\end{align*}
\end{lemma}
\begin{proof}
As $s\downarrow t$, both $\bs y_s$ and $\bs y_{s-}$ converge to
$\bs y_t$. Similarly, as $s\uparrow t$, both converge to
$\bs y_{t-}$. The result follows from the closedness of the graph of
$F$, which is a consequence of the continuity of $f$.
\end{proof}

\begin{lemma}
  \label{lem:small-h}
  Let $f$ be continuous. Let $\bs y$ be a càdlàg process on $\bR^{dn}$. For every $t>0$,
  $$
  \limsup_{h\downarrow 0} \ \mathrm{conv}\left(\bigcup_{s\in[t-h,t+h]}F (\bs y_s)\cup F(\bs y_s^-)\right) = \mathrm{conv} (F (\bs y_t)\cup F(\bs y_t^-)) 
  $$
\end{lemma}
\begin{proof}
  We prove the inclusion $\subset$, the other inclusion being trivial. Consider a collection  $(u_h)_{h>0}$ of points in $\bR^d$, converging to a point $u\in \bR^d$ as $h\downarrow 0$ along some subsequence, and such that, for every $h$,
  $$
  u_h = \sum_{l=1}^{d+1} w_{l,h} v_{l,h}\,,
  $$
  where $(w_{l,h})$ is on the $(d+1)$-dimensional simplex, and $v_{l,h} \in F (\bs y_{s_{l,h}})\cup F(\bs y_{s_{l,h}}^-)$ for some $s_{l,h}\in [t-h,t+h]$. In particular, $s_{l,h}\to t$.
  By relative compactness, one can extract a subsequence along which, for every $j$, $w_{l,h}$ and $v_{l,h}$ are converging, respectively to some $w_l$ and $v_l$.
  $(w_l)$ is also on the simplex, and $v_l\in F(\bs y_t)\cup F(\bs y_t^-)$ by Lem.~\ref{lem:lim-F-leftright}. Moreover, $u= \sum_l w_l v_l$, which concludes the proof.
\end{proof}


\begin{lemma}
\label{lem:A1-closed}
Let $f$ be continuous. For every $\delta,T>0$, define
$\cA_1^{\delta,T}:\cC_d\rightrightarrows\cD_d$ by
\[
\cA_1^{\delta,T}(x)
\eqdef
\left\{
y\in\cD_d:
\forall\,0\leq t<T,\quad
f(y_t)\leq\min_{0\leq s\leq t}f(x_s)+\delta
\right\}.
\]
Then $\graph(\cA_1^{\delta,T})$ is closed in
$\cC_d\times\cD_d$.
\end{lemma}

\begin{proof}
Let $(x^n,y^n)\to(x,y)$ in $\cC_d\times\cD_d$, with
$y^n\in\cA_1^{\delta,T}(x^n)$. Set $t<T$ and choose $T'>T$.
Skorokhod convergence provides time changes $\lambda_n$ on
$[0,T']$ such that
\[
\|\lambda_n-I\|_{T'}\to0,
\qquad
\|y^n\circ\lambda_n-y\|_{T'}\to0,
\]
an $T'$ a continuity point of $y$.
In particular, $t_n\eqdef\lambda_n(t)<T$ for all sufficiently large
$n$. Hence
\[
f(y^n_{t_n})
\leq
\min_{0\leq s\leq t_n}f(x_s^n)+\delta.
\]
The left-hand side converges to $f(y_t)$. Moreover,
$f\circ x^n\to f\circ x$ uniformly on $[0,T']$, and therefore
\[
\min_{0\leq s\leq t_n}f(x_s^n)
\longrightarrow
\min_{0\leq s\leq t}f(x_s).
\]
Passing to the limit gives
\[
f(y_t)\leq\min_{0\leq s\leq t}f(x_s)+\delta.
\]
Since $t<T$ was arbitrary, $y\in\cA_1^{\delta,T}(x)$.
\end{proof}

\begin{lemma}
\label{lem:A2-closed}
The correspondence $\cA_2:\cC_d\rightrightarrows\cD_d$ defined by
\[
\cA_2(x)
\eqdef
\left\{
y\in\cD_d:
\forall\,0\leq s<t,\quad
y_t\in\{y_s\}\cup x((s,t])
\right\}
\]
has a closed graph.
\end{lemma}

\begin{proof}
Let $(x^n,y^n)\to(x,y)$ in $\cC_d\times\cD_d$, with
$y^n\in\cA_2(x^n)$ for every $n$. Set $0\leq s<t$. If $y_t=y_s$,
there is nothing to prove. Otherwise, by right continuity of $y$ at
$s$, there exists $r\in(s,t)$ such that $y_r\neq y_t$.

Choose $ T>t$ at which $y$ is continuous. The restrictions of
$y^n$ to $[0, T]$ then converge to that of $y$ in the
Skorokhod topology. Hence, there exist increasing bijections
$\lambda_n:[0, T]\to[0, T]$ such that
\[
\|\lambda_n-I\|_{ T}\longrightarrow0,
\qquad
\|y^n\circ\lambda_n-y\|_{ T}\longrightarrow0.
\]
Since $y_r\neq y_t$, for all sufficiently large $n$,
\[
y^n_{\lambda_n(r)}\neq y^n_{\lambda_n(t)}.
\]
Using $y^n\in\cA_2(x^n)$ and
$\lambda_n(r)<\lambda_n(t)$, there exists
$u_n\in(\lambda_n(r),\lambda_n(t)]$ such that
\[
y^n_{\lambda_n(t)}=x^n_{u_n}.
\]

After extracting a subsequence, we may assume that $u_n\to u$ for some
$u\in[0, T]$. Since
\[
\lambda_n(r)<u_n\leq\lambda_n(t),
\]
we obtain $r\leq u\leq t$, and hence $u\in(s,t]$. Moreover,
\[
\|x^n_{u_n}-x_u\|
\leq
\|x^n-x\|_{ T}
+\|x_{u_n}-x_u\|
\longrightarrow0,
\]
where we used the uniform convergence of $x^n$ to $x$ on
$[0, T]$ and the continuity of $x$. Since
$y^n_{\lambda_n(t)}\to y_t$, it follows that $y_t=x_u$, and therefore
\[
y_t\in x((s,t]).
\]
Thus, for every $0\leq s<t$,
\[
y_t\in\{y_s\}\cup x((s,t]),
\]
which proves that $y\in\cA_2(x)$. Hence $\graph(\cA_2)$ is closed.
\end{proof}

\section{Conclusions}

In this paper, we established the continuous-time limit of momentum-free PSO through weak convergence
of sequences of interpolated processes with diminishing step-size. 
Our work specifically tackles the challenges of the irregular personal-best and global-best processes, that
have previously been replaced by smooth approximations in other works concerning continuous-time dynamics of PSO.

Our work opens the path to novel theoretical analyses of PSO. This continuous-time framework can be used to study the exploration effect of the Brownian noise, the exploitation mechanisms of exact personal- and global-best attractors, and hitting times of the algorithm for a 
fixed number of particles.


Furthermore, the analytical approach developed in this paper is not strictly limited to PSO. 
The mathematical tools used to handle the discontinuous, path-dependent selection rules can be extended
to a broader class of interacting particle algorithms with exact minimizers over particle trajectories. 
In particular, this framework could be adapted to more general 
drift functions $b_i(\bs x, \bs y, \theta)$ and diffusion coefficients $\sigma_i(\bs x, \bs y)$, provided they are sufficiently regular
 to ensure local boundedness of the dynamics. Additionally, the 
 state space can naturally be augmented to include auxiliary variables, such as velocity, to handle PSO with momentum.

\subsection*{Declaration on the use of generative AI tools}

During the preparation of this work, the authors used ChatGPT in order to make writing improvements and suggest minor simplifications in the proofs. The original ideas and core mathematical reasoning were developed entirely by the authors. After using AI tools, the authors reviewed and edited the content as needed and take full responsibility for the content of the published article.

\subsection*{Acknowledgements} This work benefited from Hi! PARIS and State funding managed by the French National Research Agency (ANR) under the France 2030 program, reference ANR-23-IACL-0005.


\bibliographystyle{elsarticle-num}
\bibliography{literature}

\end{document}